\documentclass[12pt]{article}
\usepackage[utf8]{inputenc}
 \usepackage{a4}
 \usepackage{amsmath}
 \usepackage{amssymb}
\usepackage{amsfonts}
\usepackage{amsthm}
\usepackage{graphicx}
\newtheorem{thm}{Theorem}[section]
\newtheorem{cor}[thm]{Corollary}
\newtheorem{lem}[thm]{Lemma}
\newtheorem{prop}[thm]{Proposition}
\newtheorem{defn}[thm]{Definition}
\newtheorem{rem}[thm]{Remark}
\newtheorem{ex}[thm]{Example}

\def\tr{\mbox{tr}\,}

\newcommand{\ww}[1]{\rho^{[2]}_{#1}}
\newcommand{\dd}[1]{\tau^{[2]}_{#1}}
\newcommand{\w}[1]{\rho^{[1]}_{#1}}

\long\def\comment#1{{}}

\title{On the order of indeterminate moment problems}

\author{Christian Berg and Ryszard Szwarc \thanks{The first author acknowledges
     support by grant 10-083122 from The Danish 
    Council for Independent Research $|$ Natural Sciences}}

\begin{document}
\maketitle

\begin{abstract}
For an indeterminate moment problem we denote the orthonormal
polynomials by $P_n$. We study the relation between the growth of the
function $P(z)=(\sum_{n=0}^\infty|P_n(z)|^2)^{1/2}$ and summability properties
of the sequence $(P_n(z))$. Under certain assumptions on the
recurrence coefficients  
 from the  three term recurrence relation
$ zP_n(z)=b_nP_{n+1}(z)+a_nP_n(z)+b_{n-1}P_{n-1}(z)$,
we show that the function $P$ is of order
$\alpha$ with $0<\alpha<1$, if and only if the sequence $(P_n(z))$ is absolutely summable to any power greater than $2\alpha$. Furthermore, the  
order $\alpha$ is equal to the exponent of
convergence of the sequence $(b_n)$.
 Similar results are obtained for logarithmic order and for more general types of slow growth. To prove these results we introduce a concept of an order function and its dual. 

We also relate the order of $P$ with the order of certain  entire functions defined in terms of the moments or the leading coefficient of $P_n$. 
\end{abstract}



\section{Introduction and results}\label{sec-intro}

Stieltjes discovered the indeterminate moment problem in the memoir \cite{St} from 1894, and one can follow his discoveries in the correspondence
 with Hermite, cf. \cite{Be95}. Stieltjes only considered distribution functions on the half-line $[0,\infty)$ corresponding to what
 is now called the Stieltjes moment problem. It took about 25 years before
Hamburger, Nevanlinna and Marcel Riesz laid the foundation of the Hamburger moment problem described by \eqref{eq:Ha}. Nevanlinna proved 
the Nevanlinna parametrization
 of the full set of solutions to the Hamburger moment problem. Using the  four entire functions $A,B,C,D$,
obtained from \eqref{eq:ABCD} by letting $n\to\infty$, any solution to the moment problem can be described  via a universal parameter space, namely
the one-point compactification of the space of Pick functions. Nevanlinna  also pointed out what is now called the Nevanlinna extremal solutions
 corresponding to the degenerate Pick functions, which are a real constant
 or infinity. Since the same solutions appear in spectral theory for  self-adjoint extensions of Jacobi-matrices, Simon \cite{Si} proposed to call
 them von-Neumann solutions. The classical monographs describing the Nevanlinna parametrization are \cite{Ak},\cite{S:T},\cite{Sto}.
 None of these treatises contain
a fully calculated example with concrete functions $A,B,C,D$. Although it was well known that the zeros of $B,D$ interlace and similarly with $A,C$,
 nobody seem to have noticed that  these functions have the same growth properties before it was done in \cite{B:P}. In that paper it was proved
  that the four entire functions $A,B,C,D$ as well as $P,Q$ from Theorem~\ref{thm:indet} have the same order and type
called the {\it order $\rho$ and type $\tau$ of the indeterminate moment problem}.
Long before, Marcel Riesz had proved the deep result that  $A,B,C,D$ are of minimal exponential type,
i.e., that $0\le\rho\le 1$ and if $\rho=1$, then $\tau=0$, cf. \cite[p. 56]{Ak}. 

 I a series of papers in the beginning of the 1990'ies, Ismail-Masson \cite{I:M}, Chihara-Ismail \cite{Chi:I}, Berg-Valent \cite{B:V} calculated a number of examples. One source of indeterminate moment problems is $q$-series, cf. \cite{G:R}, and formulas of Ramanujan, see \cite{Askey}.
 The indeterminate  
moment problems within the $q$-Askey scheme were identified  by Christiansen in \cite{Chr}. All these moment problems     
have order zero, and in  Ismail \cite{Ism} it was conjectured that  $A,B,C,D$ should have the
same growth properties on a more refined scala than ordinary order. This was proved in 
\cite{B:P:H}, by the introduction of a refined scale called logarithmic order and type, so we can speak about 
logarithmic order $\rho^{[1]}$ and logarithmic type
$\tau^{[1]}$ of a moment problem of order zero. In \cite{Ped} it was proved that if $(\rho,\tau)$ or $(\w{},\tau^{[1]})$ are prescribed, then there exist  
indeterminate moment problems with these (logarithmic) orders and
types. In Ramis \cite{Ra} the notion of logarithmic order and type
appears for entire solutions to $q$-difference equations.

The main achievement of the present paper is that we present some conditions on the coefficients $(a_n),(b_n)$ of the three term recurrence relation
\eqref{eq:3t}, such that when these hold, then summability properties of the sequence $(P_n^2(z))$ and order properties of the moment problem are
 equivalent. Furthermore, the order as well as the logarithmic order of the moment problem can be calculated from the growth properties of the sequence
$(b_n)$.

 These conditions are of two different types. There is a {\it regularity condition} that $(b_n)$ is either log-convex eventually
 or log-concave eventually, cf.  \eqref{eq:logconvex} or \eqref{eq:logconcave}, and a {\it growth condition} \eqref{eq:finalber}.

The last condition is also necessary in the symmetric case $a_n=0$ because of Carleman's condition.

We shall now give a more detailed introduction to the content.

 Consider a normalized Hamburger moment sequence $(s_n)$ given as
\begin{equation}\label{eq:Ha}
s_n=\int_{-\infty}^\infty x^n\,d\mu(x),\quad n\ge 0,
\end{equation}
where $\mu$ is a probability measure with infinite support and moments
of any order.

 Denote the corresponding orthonormal
polynomials by $P_n(z)$ and those of the second kind by $Q_n(z)$,
following the notation and terminology of \cite{Ak}.
These polynomials satisfy a three term recurrence relation of the form
\begin{equation}\label{eq:3t}
z r_n(z)=b_nr_{n+1}(z)+a_nr_n(z)+b_{n-1}r_{n-1}(z),\quad n\ge 0,
\end{equation}
where $a_n\in\mathbb R,b_n>0$ for $n\ge 0$ and $b_{-1}=1$, and
 with the initial conditions
$P_0(z)=1,\ P_{-1}(z)=0$ and $Q_0(z)=0,\ Q_{-1}(z)=-1$.

The following polynomials will be used, cf. \cite[p.14]{Ak}
\begin{equation}\label{eq:ABCD}
\begin{array}{llll}
 A_n(z)=z\sum_{k=0}^{n-1} Q_k(0)Q_k(z),\\
B_n(z)=-1+z\sum_{k=0}^{n-1} Q_k(0)P_k(z),\\
C_n(z)=1+z\sum_{k=0}^{n-1} P_k(0)Q_k(z),\\
D_n(z)=z\sum_{k=0}^{n-1} P_k(0)P_k(z).
\end{array}
\end{equation}

We need the coefficients of the orthonormal polynomials 
\begin{equation}\label{eq:OP1}
P_n(x)=\sum_{k=0}^nb_{k,n}x^k,
\end{equation}
and by \eqref{eq:3t} we have
\begin{equation}\label{eq:lead}
b_{n,n}=1/(b_0b_1\cdots b_{n-1})>0.
\end{equation}

The indeterminate case is characterized by the equivalent conditions
in the following result, cf. \cite[Section 1.3]{Ak}.

\begin{thm}\label{thm:indet} For $(s_n)$ as in \eqref{eq:Ha} the
  following conditions are equivalent:
\begin{enumerate}
\item[(i)]$\sum_{n=0}^\infty \left(P_n^2(0)+Q_n^2(0)\right)<\infty,$
\item[(ii)] $ P(z)=\left(\sum_{n=0}^\infty |P_n(z)|^2\right)^{1/2}<\infty,\quad z\in\mathbb C.$
\end{enumerate}
If (i) and (ii) hold (the indeterminate case), then 
$Q(z)=\left(\sum_{n=0}^\infty |Q_n(z)|^2\right)^{1/2}<\infty$ for $z\in\mathbb C$,
 and $P,Q$ are continuous functions.
\end{thm}

 Concerning order and type as well as logarithmic order and type of  an (entire) function, we refer to
 Section~\ref{sec-pre}, but we warn the reader that the logarithmic order treated in this paper differs from the logarithmic order of \cite{B:P:H} by 
 subtracting 1.

Our first main result extends Theorem~\ref{thm:indet}.
For $0<\alpha$ we consider the complex linear sequence  space 
$$
\ell^{\alpha}=\{(x_n)|\sum_{n=0}^\infty |x_n|^\alpha<\infty\}.
$$

\begin{thm}\label{thm:order1intro} For a moment problem and  $0<\alpha \le 1$ the following
  conditions are equivalent:
\begin{enumerate}
\item[(i)] $(P_n^2(0)),(Q_n^2(0))\in\ell^{\alpha}$,
\item[(ii)] $(P_n^2(z)),(Q_n^2(z))\in\ell^{\alpha}$ for all
  $z\in\mathbb C$.
\end{enumerate}
If the conditions are satisfied, the moment problem is indeterminate and the two series indicated in (ii) converge
uniformly on compact subsets of $\mathbb C$. Furthermore, $(1/b_n)\in\ell^{\alpha}$ and
\begin{equation}\label{eq:Porderintro}
P(z)\le C\exp(K|z|^\alpha),
\end{equation}
where
\begin{equation}\label{eq:CKintro}
C=\left(\sum_{n=0}^\infty(P_n^2(0)+Q_n^2(0))\right)^{1/2},\quad K=\frac{1}{\alpha}\sum_{n=0}^\infty(|P_n(0)|^{2\alpha}+|Q_n(0)|^{2\alpha}).
\end{equation}

In particular the moment problem has order $\rho\le\alpha$, and if the
order is $\alpha$, then the type $\tau\le K$.
\end{thm} 

\begin{rem}\label{thm:remaboutalpha} {\rm The main point in
    Theorem~\ref{thm:order1intro} is that (i) or (ii) imply
    \eqref{eq:Porderintro}. The equivalence between (i) and (ii) is in
    principle known, since it can easily be deduced from formula
    [1.23a] in Akhiezer \cite{Ak}. The theorem is proved in  Section~\ref{sec-ber}
 as Theorem~\ref{thm:order1}.
} 
\end{rem}

For an indeterminate moment problem the recurrence coefficients $(b_n)$ satisfy $\sum 1/b_n<\infty$ by Carleman's Theorem. On the other hand the condition
$\sum 1/b_n<\infty$ is not sufficient for indeterminacy, but if a condition of log-concavity is added, then indeterminacy holds by a result of
 Berezanski\u{\i} \cite{Be}, see \cite[p.26]{Ak}. This result is extended in Section~\ref{sec-ber} to include log-convexity,
  leading to the following main result, which is an almost converse of
  Theorem~\ref{thm:order1intro} in the sense that
  \eqref{eq:Porderintro} implies (i) and (ii) except for an
  $\varepsilon$, but under additional assumptions of the recurrence coefficients.

\begin{thm}\label{thm:beralphaintro} Assume that the coefficients of \eqref{eq:3t}  satisfy 
\begin{equation}\label{eq:origber}
\sum_{n=1}^\infty \frac{1+|a_n|}{b_{n-1}}<\infty,
\end{equation}
and that either \eqref{eq:logconvex} or \eqref{eq:logconcave} holds.
Assume in addition that $P$ satisfies 
$$
P(z)\le C\exp(K|z|^\alpha)
$$
for some $\alpha$ such that $0<\alpha<1$ and suitable constants $C,K>0$. 
 
Then
\begin{equation}\label{eq:Pnalphaintro}
1/b_n, P_n^2(0),Q_n^2(0)=O(n^{-1/\alpha}),
\end{equation}
so in particular $(1/b_n),(P_n^2(0)),(Q_n^2(0))\in \ell^{\alpha+\varepsilon}$ for any $\varepsilon>0$.
\end{thm}

Theorem~\ref{thm:beralphaintro} is proved as Theorem~\ref{thm:beralpha}, where we have replaced condition \eqref{eq:origber} by the slightly weaker
condition \eqref{eq:finalber}.
Under the same assumptions we prove in Theorem~\ref{thm:beralpha1} that the order of the moment problem is equal to the convergence exponent of the sequence $(b_n)$. In case of order zero it is also possible to characterize the logarithmic order of the moment problem as the convergence exponent of the sequence $(\log b_n)$, cf. Theorem~\ref{thm:Berlogorder}.

In Section~\ref{sec-ord} the results of Theorem~\ref{thm:order1intro} and of Theorem~\ref{thm:beralphaintro}
are extended to more general types of growth, based on a notion of an order function and its dual. See
  Theorem~\ref{thm:beta1} and Theorem~\ref{thm:order2}.  

In Section~\ref{sec-double}  we focus on order functions of the form $\alpha(r)=(\log\log r)^{\alpha}$, which lead to the concept of double logarithmic order and type, giving a refined classification of entire functions and moment problems of logarithmic order 0.
The six functions $A,B,C,D,P,Q$ have the same double logarithmic order and type called the double logarithmic order $\ww{}$ and type $\dd{}$ of the
moment problem.

 We establish a number of formulas
expressing the double logarithmic order and type of an entire function in terms of the coefficients in the power series expansion and the zero counting
function. The proof of these results are given in the Appendix.

For an indeterminate moment problem the numbers 
$$
c_k=\left(\sum_{n=k}^\infty b_{k,n}^2\right)^{1/2}
$$
were studied by the authors in \cite{B:S}, and $c_k$ tends to zero so quickly that
$$
\Phi(z)=\sum_{k=0}^\infty c_kz^k
$$
determines an entire function of minimal exponential type. We study this function in Section~\ref{sec-phi} and prove that $\Phi$ has the same order and type as
 the moment problem, and if the common order is zero, then $\Phi$ has the same logarithmic order and type as the moment problem. 
This is extended to double logarithmic order and type in Section~\ref{sec-double}.

In Section~\ref{sec-lif}  we revisit a paper  
\cite{L} by Liv\v sic, where it was proved that the function
$$
F(z)=\sum_{n=0}^\infty \frac{z^{2n}}{s_{2n}}
$$
has  order  less than or equal to the order of the entire function 
$$
B(z)=-1+z\sum_{k=0}^\infty Q_k(0)P_k(z).
$$
We give a another proof of this result and extend it to logarithmic and double logarithmic order, using results about $\Phi$. It seems to be unknown whether the order of $F$
 is always equal to the order of the moment problem. We prove in Theorem~\ref{thm:=inLif1} that this the case,
if the recurrence coefficients satisfy the conditions
of Theorem~\ref{thm:ber}, and at the same time it turns out that the entire function
$$
H(z)=\sum_{n=0}^\infty b_{n,n}z^n,
$$
where $b_{n,n}$ is the leading coefficient of $P_n$, cf. \eqref{eq:OP1}, also has this common order.

\section{Preliminaries}\label{sec-pre}

For a continuous function $f:\mathbb C\to\mathbb C$ we define the
maximum modulus $M_f:[0,\infty[\to[0,\infty[$ by
$$
M_f(r)=\max_{|z|\le r}|f(z)|.
$$
The order $\rho_f$ of $f$ is defined as the infimum of the numbers
$\alpha>0$ for which there exists a majorization of the form
$$
\log M_f(r)\le_{\text{\tiny as}} r^\alpha,
$$
where we use a notation inspired by \cite{Le}, meaning that the above inequality
holds for $r$ sufficiently large. We will only discuss these concepts
for unbounded functions $f$, so that $\log M_f(r)$ is positive for $r$
sufficiently large.

It is easy to see that
$$
\rho_f=\limsup_{r\to\infty}\frac{\log\log M_f(r)}{\log r}.
$$

If $0<\rho_f<\infty$ we define the type $\tau_f$ of $f$ as
$$
\tau_f=\inf \{ c >0\, | \, \log M_f(r)\leq_{\text{\tiny as}} c r^{\rho_f}\},
$$
and we have
$$
\tau_f=\limsup_{r\to\infty}\frac{\log M_f(r)}{r^{\rho_f}}.
$$

The logarithmic order as defined in \cite{B:P:H},\cite{Ped} is a
number in the interval $[1,\infty]$,  and the functions studied in
Ramis \cite{Ra} are of logarithmic order 2. A detailed study of
meromorphic functions of finite logarithmic order has been published
in Chern \cite{C}.  

 We find it appropriate to renormalize this 
definition by subtracting 1, so the new logarithmic order of this paper belongs to
the interval $[0,\infty]$. This will simplify certain formulas,
which will correspond to formulas for the double logarithmic order
developed in Section~\ref{sec-double}.

For an unbounded continuous function $f$ we define the {\em
  logarithmic order}  $\w{f}$ as
$$
\w{f} = \inf \{ \alpha >0\, | \, \log M_f(r)\le_{\text{\tiny as}}(\log r )^{\alpha+1}\,
\}
= \inf \{ \alpha >0\, | \, M_f(r)\le_{\text{\tiny as}} r^{(\log r)^\alpha}\,
\},
$$
where $\w{f}=\infty$, if there are no $\alpha>0$ satisfying the asymptotic inequality. Of course $\w{f}<\infty$ is only possible for functions of order 0.

Note that  an entire function $f$ satisfying $\log
M_f(r)\le_{\text{\tiny as}} (\log r)^\alpha$ for some $\alpha<1$ is
constant by the Cauchy estimate
$$
\frac{|f^{(n)}(0)|}{n!}\le \frac{M_f(r)}{r^n}.
$$

It is easy to obtain that
$$
\w{f} = \limsup_{r\to \infty}\frac{\log \log M_f(r)}{\log \log r}-1.
$$

When $\w{f}<\infty$ we define the {\em logarithmic type} $\tau_f^{[1]}$ as
\begin{eqnarray*}
\tau_f^{[1]} &=& \inf \{ c >0\, | \, \log M_f(r)\leq _{\text{\tiny
    as}} c(\log r )^{\w{f}+1}\, \}\\
&=&
 \inf \{ c >0\, | \, M_f(r)\leq_{\text{\tiny as}} r^{c(\log r)^{\w{f}}}\,\},
\end{eqnarray*}
and it is readily found  that
$$
\tau_f^{[1]} = \limsup_{r\to \infty}\frac{\log M_f(r)}{(\log r)^{\w{f}+1}}.
$$
An entire function $f$ satisfying  $\w{f}=0$ and $\tau_f^{[1]}<\infty$
is necessarily a polynomial of degree $\le \tau_f^{[1]}$.

The shifted moment problem is 
associated with the cut off sequences $(a_{n+1})$ and $(b_{n+1})$ from
\eqref{eq:3t}. In
terms of Jacobi matrices, the Jacobi matrix $J_s$ of the shifted
problem is obtained from the original Jacobi matrix $J$ by deleting
the first row and column. It is well-known that a moment problem and
the shifted one are either both determinate or both indeterminate. 
If indeterminacy holds, Pedersen \cite{Ped0} studied the relationship
between the $A,B,C,D$-functions of the two problems and deduced
that the shifted moment problem has the same order and type as
the original problem. We mention that the $P$-function of the shifted
problem equals $b_0Q(z)$. This equation shows that the two
problems have the same logarithmic order and type in case the common
order is zero.

By repetition, the $N$-times shifted problem is then indeterminate
with the same growth properties as the original problem. This means
that it is the large $n$ behaviour of the recurrence coefficients
which determine the order and type of an indeterminate moment
problem. This is in contrast to the behaviour of the moments, where a
modification of the zero'th moment can change an indeterminate moment
problem to a determinate one, see e.g. \cite[Section 5]{B:S}.  

In the indeterminate case we can define an entire function of two
complex variables
\begin{equation}\label{eq:kernel}
K(z,w)=\sum_{n=0}^\infty P_n(z)P_n(w)=\sum_{j,k=0}^\infty a_{j,k}z^jw^k,
\end{equation}
called the {\it reproducing kernel} of the moment problem,
and we collect the coefficients of the power series as the symmetric matrix
$\mathcal A=(a_{j,k})$ given by
\begin{equation}\label{eq:ajk}
a_{j,k}=\sum_{n=\max(j,k)}^\infty b_{j,n}b_{k,n}.
\end{equation}
It was proved in \cite{B:S} that the series \eqref{eq:ajk} is absolutely convergent
and that the matrix $\mathcal A$ is of trace class with 
$$
\tr(\mathcal A)=\rho_0,
$$
where $\rho_0$ is given by
\begin{equation}\label{eq:rho}
\rho_0=\frac{1}{2\pi}\int_0^{2\pi}K(e^{it},e^{-it})\,dt=
\frac{1}{2\pi}\int_0^{2\pi}P^2(e^{it})\,dt<\infty.
\end{equation}

Define
\begin{equation}\label{eq:cn}
c_k=\sqrt{a_{k,k}}=\left(\sum_{n=k}^\infty b_{k,n}^2\right)^{1/2}.
\end{equation}

From \eqref{eq:OP1} we have
\begin{equation}\label{eq:OP2}
b_{k,n}=\frac{1}{2\pi i}\int_{|z|=r}P_n(z)z^{-(k+1)}\,dz
=r^{-k}\frac{1}{2\pi}\int_0^{2\pi} P_n(re^{it})e^{-ikt}\,dt.
\end{equation}

By 
\eqref{eq:OP2} and by Parseval's identity
we have for $r>0$
\begin{equation}\label{eq:trzy}
\sum_{k=0}^\infty r^{2k}\sum_{n=k}^\infty |b_{k,n}|^2=\sum_{n=0}^\infty\sum_{k=0}^n r^{2k}|b_{k,n}|^2=\sum_{n=0}^\infty
\frac{1}{2\pi}\int_0^{2\pi} |P_n(re^{it})|^2\,dt,
 \end{equation}
hence
\begin{equation}\label{eq:ckP}
\sum_{k=0}^\infty r^{2k}c_k^2=\frac{1}{2\pi}\int_0^{2\pi} P^2(re^{it})\,dt,
\end{equation}
an identity already exploited in \cite{B:S}.

\section{The order and type of $\Phi$}\label{sec-phi}

The heading refers to the function
\begin{eqnarray}\label{eq:Phi}
\Phi(z)=\sum_{k=0}^\infty c_kz^k,
\end{eqnarray}
where $c_k$ is defined in \eqref{eq:cn}. By \cite[Prop. 4.2]{B:S} we
know that $\lim_{k\to\infty}k\root{k}\of{c_k}=0$, which shows that $\Phi$
is an entire function of minimal exponential type.

\begin{thm}\label{thm:ord1} The order and type of $\Phi$ are equal to the order $\rho$
  and  type $\tau$  of the moment problem.
\end{thm}

{\it Proof.} By \eqref{eq:OP1}  and \eqref{eq:ajk} we  have
\begin{eqnarray}\label{eq:D}
D(z)= z\sum_{k=0}^\infty P_k(0)P_k(z)=z\sum_{k=0}^\infty
b_{0,k}\sum_{j=0}^k b_{j,k}z^j
= z\sum_{j=0}^\infty a_{j,0}z^j.
\end{eqnarray}

Therefore,
\begin{eqnarray}\label{eq:|D|}
|D(z)|\le |z|\sum_{j=0}^\infty |a_{j,0}||z|^j\le
c_0|z|\sum_{j=0}^\infty
c_j|z|^j,
\end{eqnarray}
where we used $|a_{j,k}|\le c_jc_k$. This leads to the following
inequality for the maximum moduli
 \begin{equation}\label{eq:D-Phi}
 M_D(r)\le c_0 r M_\Phi(r),
\end{equation}
from which we clearly get $\rho=\rho_D\le \rho_\Phi$.

Since 
$\rho_P=\rho$ (the order of the moment problem), we get for any $\varepsilon>0$
$$
P(re^{i\theta})\le \exp(r^{\rho+\varepsilon})
\quad\mbox{for}\;\;r\ge R(\varepsilon).
$$
Defining
\begin{eqnarray}\label{eq:Psi}
\Psi(z)=\sum_{k=0}^\infty c_k^2z^{2k},
\end{eqnarray}
we get by \eqref{eq:ckP}
$$
M_\Psi(r)=\sum_{k=0}^\infty c_k^2r^{2k}\le\exp(2r^{\rho+\varepsilon})
\le \exp(r^{\rho+2\varepsilon})\quad\mbox{for}\;\;r\ge
\max(R(\varepsilon),2^{1/\varepsilon}),
$$
hence $\rho_\Psi\le \rho+2\varepsilon$ and finally $\rho_\Psi\le
\rho$.

However, $\rho_\Psi=\rho_\Phi$ because for an entire function
 $f(z)=\sum_{n=0}^\infty a_n z^n$ it is known (\cite{Le}) that
\begin{equation}\label{eq:deforder}
\rho_f=\limsup_{n\to\infty}\frac{\log n}{\log\left(\frac{1}{\root n\of
    {|a_n|}}\right)}.
\end{equation}
This shows the assertion of the theorem concerning order.

Concerning type, let us assume that the common order of the moment
problem and $\Phi$ is $\rho$, satisfying $0<\rho<\infty$ in order to
define type. For a function $f$ as above with order $\rho$, the type
$\tau_f$ can be determined as
\begin{eqnarray}\label{eq:type} 
\tau_f=\frac{1}{e\rho}
\limsup_{n\to\infty} \left(n |a_n|^{\rho/n}\right),
\end{eqnarray}
cf. \cite{Le}.

From \eqref{eq:D-Phi} we  get $\tau=\tau_D\le \tau_\Phi$, where
$\tau$ is the type of the moment problem.

Since $P$ has type $\tau$, we know that $|P(re^{i\theta})|\le
e^{(\tau+\varepsilon)r^\rho}$ for $r$ sufficiently large depending on
$\varepsilon>0$, hence by \eqref{eq:ckP}
$$
M_\Psi(r)=\sum_{k=0}^\infty
c_k^2r^{2k}\le\exp(2(\tau+\varepsilon)r^{\rho}),
$$
and we conclude that $\tau_\Psi\le 2\tau$. Fortunately
$\tau_\Psi=2\tau_\Phi$, as is easily seen from 
\eqref{eq:type}, so we get $\tau_\Phi\le \tau$, and the assertion about
type has been proved.
$\quad\square$

\begin{thm}\label{thm:ord2} Suppose the order of the moment problem is zero. Then $\Phi$
  has the same logarithmic order $\rho^{[1]}$ and type $\tau^{[1]}$ as the moment problem.
\end{thm}

\begin{proof} The logarithmic order $\w{f}$ of an entire
function $f=\sum_0^\infty a_nz^n$ of order zero can be calculated as
\begin{eqnarray}\label{eq:logorder}
\w{f} 
= \limsup_{n\to \infty}\frac{\log n}{\log \log\left(
    \frac{1}{\sqrt[n]{|a_n|}}\right)},
\end{eqnarray}
cf. \cite{B:P:H}.
From \eqref{eq:D-Phi} we want to see that 
$\rho^{[1]}=\w{D}\le\w{\Phi}$. This is
clear if $\w{\Phi}=\infty$, so assume it to be finite. For any
$\varepsilon>0$ we have for $r$ sufficiently large
$$
M_D(r)\le c_0rr^{(\log r)^{\w{\Phi}+\varepsilon}}\le
r^{(\log r)^{\w{\Phi}+2\varepsilon}},
$$
which gives the assertion.

We next use that for given $\varepsilon>0$ we have for $r$
sufficiently large
$$
P(re^{i\theta})\le r^{(\log r)^{\rho^{[1]}+\varepsilon}},
$$
which by \eqref{eq:ckP} yields 
$$
M_\Psi(r)\le r^{2(\log r)^{\rho^{[1]}+\varepsilon}}\le_{\text{\tiny as}} r^{(\log r)^{\rho^{[1]}+2\varepsilon}},
$$
hence $\w{\Psi}\le \rho^{[1]}$. From \eqref{eq:logorder} we see that  $\w{\Phi}=\w{\Psi}$, hence
$\rho^{[1]} = \w{\Phi}$.

We next assume that the common value $\rho^{[1]}$ of the logarithmic
order is a finite number $>0$. (Transcendental function of logarithmic
order 0 have necessarily logarithmic type $\infty$.)  We shall show
that $\tau^{[1]}=\tau_\Phi^{[1]}$ and recall that the logarithmic type
$\tau_f^{[1]}$ of a function $f=\sum_0^\infty a_nz^n$ with
logarithmic order $0<\rho^{[1]}<\infty$ 
is given by the formula, cf. \cite{B:P:H},
\begin{eqnarray}\label{eq:logtype}
\tau_f^{[1]}=\frac{(\rho^{[1]})^{\rho^{[1]}}}{(\rho^{[1]}+1)^{\rho^{[1]}+1}}\limsup_{n\to
  \infty}\frac{n}{\left(\log \frac{1}{\sqrt[n]{|a_n|}}\right)^{\rho^{[1]}}}.
\end{eqnarray}
Again it is clear that $\tau_\Psi^{[1]}=2\tau_\Phi^{[1]}$, and from 
\eqref{eq:D-Phi} we get $\tau^{[1]}\le \tau_\Phi^{[1]}$, while  
\eqref{eq:ckP} leads to $\tau_\Psi^{[1]}\le 2\tau^{[1]}$. This finally
gives $\tau^{[1]}=\tau_\Phi^{[1]}$.
\end{proof}

\section{Berezanski\u{\i}'s method}\label{sec-ber}

We are going to use and extend a method due to Berezanski\u{\i}
\cite{Be} giving a sufficient condition for indeterminacy. The method
is explained in \cite[p.26]{Ak}. Berezanski\u{\i} treated the case below of log-concavity.

\begin{lem}\label{thm:inc} Let $b_n>0,n\ge 0$ satisfy
\begin{equation}\label{eq:conv}
\sup_{n\ge 0} b_n = \infty
\end{equation}
and either
\begin{equation}\label{eq:logconvex}
\mbox{log-convexity:}\quad\quad b_n^2\le b_{n-1}b_{n+1},\quad n\ge n_0,
\end{equation}
or
\begin{equation}\label{eq:logconcave}
\mbox{log-concavity:}\quad\quad b_n^2\ge b_{n-1}b_{n+1},\quad n\ge n_0.
\end{equation}
Then $(b_n)$ is eventually strictly increasing to infinity. 
\end{lem}

\begin{proof} Suppose first that \eqref{eq:logconvex} holds. For $n\ge n_0$, $b_{n+1}/b_n$ is increasing, say to
$\lambda\le \infty$. If $\lambda\le 1$, then $b_n$ is decreasing for
$n\ge n_0$ in contradiction to \eqref{eq:conv}. Therefore
$1<\lambda\le\infty$ and for any $1<\lambda_0<\lambda$ we have
$b_{n+1}\ge \lambda_0 b_n$ for $n$ sufficiently large.

If \eqref{eq:logconcave} holds, then $b_{n+1}/b_n$ is decreasing for $n\ge
n_0$, say to $\lambda\ge 0$. If $\lambda<1$ then $\sum b_n<\infty$ in
contradiction to \eqref{eq:conv}. Therefore $\lambda\ge 1$ and finally
$b_{n+1}\ge b_n$ for $n\ge n_0$. If $b_n=b_{n-1}$ for some $n>n_0$, then 
\eqref{eq:logconcave} implies $b_n\ge b_{n+1}$, hence $b_n=b_{n+1}$, so $(b_n)$ is eventually constant in contradiction to
\eqref{eq:conv}.
\end{proof}

\begin{thm}[Berezanski\u{\i}]
\label{thm:ber}
Assume that the coefficients of \eqref{eq:3t}  satisfy
\begin{equation}\label{eq:finalber}
\sum_{n=1}^\infty \frac{1+|a_n|}{\sqrt{b_nb_{n-1}}}<\infty,
\end{equation}
and that either \eqref{eq:logconvex} or \eqref{eq:logconcave} holds.
 \footnote{In \cite{Ak} it is assumed
    that $|a_n|\le M$, $\sum 1/b_n<\infty$ and that
    \eqref{eq:logconcave} holds. The assertion \eqref{eq:mino} is not discussed.}

For any non-trivial solution $(r_n)$ of \eqref{eq:3t} there exists a constant $c$, depending on the $a_n,b_n$ and the initial conditions
 $(r_0,r_{-1})\neq(0,0)$
 but independent of $z$,
 such that
\begin{equation}\label{eq:berez}
\sqrt{b_{n-1}}|r_n(z)|\le c\;\Pi(|z|),\quad \Pi(z)=\prod_{k=0}^\infty\left (1+
  \frac{z}{b_{k-1}}\right ),\quad n\ge 0,
\end{equation}
and there exists a constant $K_z > 0$ for $z\in\mathbb C$ such that
\begin{equation}\label{eq:mino}
\max\{|r_n(z)|,|r_{n+1}(z)|\} \ge \frac{K_z}{\sqrt{b_{n+1}}},\quad n\ge 0.
\end{equation}
In particular,
\begin{equation}\label{eq:spec}
P_n^2(0),Q_n^2(0)=O(1/b_{n-1})
\end{equation}
and
\begin{equation}\label{eq:spec*}
 \frac{K}{b_{n+1}}\le |r_n(z)|^2+|r_{n+1}(z)|^2 \le  \frac{L}{b_{n-1}}
\end{equation}
for  suitable constants $K,L$ depending on $z$.

The moment problem is indeterminate.
\end{thm}

\begin{proof} By Lemma~\ref{thm:inc} we have $b_{n-1}< b_n$ for
  $n\ge n_1 > n_0$.
  
By the recurrence relation we get

\begin{eqnarray}\label{eq:double}
\frac{b_{n-1}}{b_n}|r_{n-1}(z)|-\frac{|z|+|a_n|}{b_n}|r_n(z)|\le
|r_{n+1}(z)|\le \nonumber\\
\frac{b_{n-1}}{b_n}|r_{n-1}(z)| + \frac{|z|+|a_n|}{b_n}|r_n(z)|.
\end{eqnarray}

Let
$$
u_n=\sqrt{b_{n-1}}|r_n(z)|,\quad v_n=\max(u_n,u_{n-1}),\quad \varepsilon_n=\frac{|z|+|a_n|}{\sqrt{b_nb_{n-1}}}.
$$
Since $(r_0,r_{-1})\neq(0,0)$ we have $v_n>0$ for $n\ge 1$, and by assumption $\varepsilon_n<1$ for $n$ sufficiently large depending on $z$, say for 
$n\ge n_z \ge n_1$.

From the second inequality in \eqref{eq:double} we  then get 
$$
u_{n+1}\le \frac{b_{n-1}}{\sqrt{b_nb_{n-2}}}u_{n-1} +  \varepsilon_n u_n \le v_n(1+\varepsilon_n), 
$$
where the last inequality requires log-convexity, assumed for $n\ge n_0$. For $n\ge n_1$ we then get
$$
v_{n+1}  \le (1+\varepsilon_n)v_n \le \left(1+\frac{|a_n|}{\sqrt{b_nb_{n-1}}}\right)\left(1+\frac{|z|}{b_{n-1}}\right)v_n.
$$

Therefore

$$
v_{n_1+n}(z)\le \prod_{k=n_1}^\infty\left(1+ \frac{|a_k|}{\sqrt{b_kb_{k-1}}}\right )
\prod_{k=n_1}^\infty\left (1+ \frac{|z|}{b_{k-1}}\right )
v_{n_1}(z),\quad n\ge 1,
$$
 and since
$$
v_{n_1}(z)\prod_{k=0}^{n_1-1}\left(1+|z|/b_{k-1}\right)^{-1}
$$
is bounded in the complex plane, we get \eqref{eq:berez} for $n>n_1$, hence for all $n$ by modifying the constant. (Remember that $b_{-1}:=1$.)

From the first inequality in \eqref{eq:double} we get for $n\ge n_z$ now using log-concavity
\begin{equation}\label{eq:down}
u_{n+1}\ge \frac{b_{n-1}}{\sqrt{b_nb_{n-2}}}u_{n-1}-\varepsilon_n u_n\ge u_{n-1}-\varepsilon_n u_n.
\end{equation}
We claim that
$$
v_{n+1}\ge (1-\varepsilon_n)v_{n},\quad n\ge n_z.
$$
This is clear if $v_n=u_n$, and if $v_n=u_{n-1}$, then $u_{n-1}\ge u_n$ so \eqref{eq:down} gives $v_{n+1}\ge u_{n+1}\ge (1-\varepsilon_n)u_{n-1}$.
For $n>n_z$ we then get
$$
v_n\ge v_{n_z}\prod_{k=n_z}^\infty (1-\varepsilon_k)>0,
$$
hence $d:=\inf_{n\ge 1}v_n>0$. Therefore either $\sqrt{b_n}|r_{n+1}(z)|\ge d$ or $\sqrt{b_{n-1}}|r_n(z)|\ge d$, which shows \eqref{eq:mino}
(even with the denominator $\sqrt{b_n}$).

We still have to prove the inequalities \eqref{eq:berez} and \eqref{eq:mino} when the assumptions of log-convexity and log-concavity are interchanged. 
To do so we change the  definition of $u_n$ to 
$u_n=\sqrt{b_n}|r_n(z)|$, and we  get from the second inequality in \eqref{eq:double}
$$
u_{n+1}\le \frac{\sqrt{b_{n-1}b_{n+1}}}{b_n}\left(u_{n-1} +\varepsilon_nu_n \right) \le 
v_n(1+\varepsilon_n), 
$$
where the last inequality requires log-concavity, assumed for $n\ge n_0$. Therefore $v_{n+1}  \le (1+\varepsilon_n)v_n$, and \eqref{eq:berez}
 follows as above.

From the first inequality in \eqref{eq:double} we similarly get
$$
u_{n+1}\ge  \frac{\sqrt{b_{n-1}b_{n+1}}}{b_n}\left(u_{n-1}-\varepsilon_nu_n\right).
$$
We now claim that in the log-convex case 
$$
v_{n+1}\ge (1-\varepsilon_n)v_{n},\quad n\ge n_z,
$$
where $n\ge n_z$ implies $\varepsilon_n<1$. This is clear if $v_n=u_n$, and if $v_n=u_{n-1}$ we have 
$u_{n-1}\ge u_n$, hence $u_{n-1}-\varepsilon_nu_n\ge (1-\varepsilon_n)u_{n-1} \ge 0$.

The proof is finished as in the first case.

From \eqref{eq:berez} we get for $z=0$ with $r_n=P_n$ and $r_n=Q_n$
that  \eqref{eq:spec} holds, and this implies indeterminacy by Theorem~\ref{thm:indet}. Finally, \eqref{eq:spec*} is obtained by combining
\eqref{eq:berez}   and \eqref{eq:mino}.
\end{proof}

\begin{rem}\label{thm:remlowerbound}
{\rm The lower bound \eqref{eq:mino} for non-real $z$ can be obtained
 differently based on the Christoffel-Darboux formula, cf. \cite[p.9]{Ak}, 
$$
({\rm Im}\,z) \, \sum_{k=0}^{n-1}|P_k(z)|^2=b_{n-1}{\rm Im}\,
[P_n(z)\overline{P_{n-1}(z)}].
$$
Hence 
$$
\frac{|{\rm Im}\,z|}{b_{n-1}} \le |P_{n-1}(z)|\,|P_n(z)|,\quad n\ge 1.
$$
Similarly, we can get the same inequality with $Q_n$ in place of $P_n$.
So far we do not need any extra assumptions on the coefficients in
the recurrence relation.

If we know that $r_n(z)$ is bounded above by $c\,\Pi(|z|)/\sqrt{b_{n-1}}$ for any solution of the recurrence relation, we immediately get 
$$
|P_{n}(z)|\ge \frac{|{\rm Im}\,z|}{c\,\Pi(|z|)\sqrt{b_{n}}} .
$$
The same is true for $Q_n$ in place of $P_n$.
}
\end{rem}

\begin{cor}\label{thm:O1n} Under the assumptions of Theorem~\ref{thm:ber} we have
$$
1/b_n,P_n^2(0),Q_n^2(0)=o(1/n).
$$
\end{cor}

\begin{proof} Since $(b_n)$ is eventually increasing by Lemma~\ref{thm:inc}, we obtain from the convergence of $\sum 1/b_n$
 that $(n/b_n)$ tends to zero. Using \eqref{eq:spec} we see that also
 $(nP_n^2(0))$ and $(nQ_n^2(0))$ tend to zero.
\end{proof}

\begin{rem}\label{thm:rem31} {\rm Note that \eqref{eq:finalber} is a weaker condition than \eqref{eq:origber} because $(b_n)$ is
 eventually increasing. 

By a theorem of Carleman, $\sum
    1/b_n=\infty$ is a sufficient condition for determinacy, and it is
    well-known that there are determinate moment problems for which
    $\sum 1/b_n<\infty$. The
    converse of Carleman's Theorem holds under the additional
 conditions of Theorem~\ref{thm:ber}.

We give next a family of examples of determinate symmetric moment
problems for which $\sum 1/b_n<\infty$. 

In the symmetric case $a_n=0$ for all $n$,
    we have $P_{2n+1}(0)=Q_{2n}(0)=0$, and it follows from
    \eqref{eq:3t} that
$$
P_{2n}(0)=(-1)^n\frac{b_0b_2\cdots b_{2n-2}}{b_1b_3\cdots
  b_{2n-1}},\quad Q_{2n+1}(0)=(-1)^n\frac{b_1b_3\cdots
  b_{2n-1}}{b_0b_2\cdots b_{2n}},
$$
so the moment problem is determinate by Theorem~\ref{thm:indet} if and
only if 
\begin{equation}\label{eq:symsum}
\sum_{n=1}^\infty \left(\frac{b_0b_2\dots b_{2n-2}}{b_1b_3\ldots
  b_{2n-1}}\right)^2 + \left(\frac{b_1b_3\ldots
  b_{2n-1}}{b_0b_2\ldots b_{2n}}\right)^2=\infty.
\end{equation}
If $\beta_n>0$ is arbitrary such that $\sum 1/\beta_n<\infty$, then
defining $b_{2n}=b_{2n+1}=\beta_n$ for $n\ge 0$, we get a symmetric
moment problem which is determinate because of \eqref{eq:symsum} since 
$$
\frac{b_0b_2\cdots b_{2n-2}}{b_1b_3\cdots
  b_{2n-1}}=1.
$$
Clearly $\sum 1/b_n<\infty$ and $(b_n)$ does not satisfy the conditions \eqref{eq:logconvex}
or \eqref{eq:logconcave}.
} 
\end{rem}

\begin{prop}\label{thm:product}
 Let $0<\alpha\le 1$, let
   $(u_n)\in\ell^{\alpha}$ be a sequence of positive numbers and define
$$
K:=\sum_{n=1}^\infty u^\alpha_n.
$$
 Then
  $$\prod_{n=1}^\infty (1+ru_n)\le \exp(\alpha^{-1}Kr^\alpha).$$  
 \end{prop}
 \begin{proof}
 The conclusion follows immediately from the inequalities below
 $$
1+ru_n\le (1+r^\alpha u_n^\alpha)^{\frac{1}{\alpha}}\le \exp(\alpha^{-1}
r^\alpha u_n^\alpha).
$$
  \end{proof}

We shall now prove 
Theorem~\ref{thm:order1intro}, and in order to make the reading easier we repeat the result:

\begin{thm}\label{thm:order1} For a moment problem and  $0<\alpha \le 1$ the following
  conditions are equivalent:
\begin{enumerate}
\item[(i)] $(P_n^2(0)),(Q_n^2(0))\in\ell^{\alpha}$,
\item[(ii)] $(P_n^2(z)),(Q_n^2(z))\in\ell^{\alpha}$ for all
  $z\in\mathbb C$.
\end{enumerate}
If the conditions are satisfied, the moment problem is indeterminate and the two series indicated in (ii) converge
uniformly on compact subsets of $\mathbb C$. Furthermore, $(1/b_n)\in\ell^\alpha$ and
\begin{equation}\label{eq:Porder}
P(z)\le C\exp(K|z|^\alpha),
\end{equation}
where
\begin{equation}\label{eq:CK}
C=\left(\sum_{n=0}^\infty(P_n^2(0)+Q_n^2(0))\right)^{1/2},\quad K=\frac{1}{\alpha}\sum_{n=0}^\infty(|P_n(0)|^{2\alpha}+|Q_n(0)|^{2\alpha}).
\end{equation}

In particular the moment problem has order $\rho\le\alpha$, and if the
order is $\alpha$, then the type $\tau\le K$.
\end{thm} 

\begin{proof} Condition (ii) is clearly stronger than condition (i). 

Assume next that (i) holds, and in particular the indeterminate case
occurs because $\ell^\alpha\subseteq \ell^1$.

Following ideas of Simon \cite{Si}, we can write \eqref{eq:ABCD} as
\begin{multline}
\begin{pmatrix}
A_{n+1}(z) & B_{n+1}(z)\\
C_{n+1}(z) & D_{n+1}(z)
\end{pmatrix} =\\
\biggl [ I + z
\begin{pmatrix}
-P_n(0)Q_n(0) & Q_n^2(0)\\
-P_n^2(0) & P_n(0)Q_n(0)
\end{pmatrix}
\biggr] \begin{pmatrix}
A_n(z) & B_n(z)\\
C_n(z) & D_n(z)
\end{pmatrix}.
\end{multline}
and evaluating the operator norm of the
matrices gives
\begin{eqnarray*}\left \| \begin{pmatrix}
A_n(z) & B_n(z)\\
C_n(z) & D_n(z)
\end{pmatrix}\right \| &\le& \prod_{k=0}^{n-1}\left [ 1+ |z|(P_k^2(0)+Q_k^2(0))\right ]
 \\ &\le& \prod_{k=0}^{n-1}\left [ 1+ |z|P_k^2(0)\right ]\   \prod_{k=0}^{n-1}\left [ 1+ |z|Q_k^2(0)\right ].
\end{eqnarray*}
 In particular we have
\begin{equation}\label{simon}
\left.
\begin{array}{l}
\sqrt{|A_n(z)|^2+|C_n(z)|^2}\\
\sqrt{|B_n(z)|^2+|D_n(z)|^2}
\end{array}
\right\}
\le \prod_{k=0}^\infty\left [ 1+ |z|P_k^2(0)\right ]\   \prod_{k=0}^\infty\left [ 1+ |z|Q_k^2(0)\right ].
\end{equation} 
By Proposition~\ref{thm:product} we obtain
\begin{equation}\label{eq:BDalpha}
\left.
\begin{array}{l}
\sqrt{|A_n(z)|^2+|C_n(z)|^2}\\
 \sqrt{|B_n(z)|^2+|D_n(z)|^2}
\end{array}
\right\}
\le \exp(\alpha^{-1}K(\alpha)|z|^\alpha),
\end{equation}
where
\begin{equation}\label{eq:Kalpha}
K(\alpha)=\sum_{k=0}^\infty (|P_k(0)|^{2\alpha}+ |Q_k(0)|^{2\alpha}).
\end{equation}
We also have (\cite[p.14]{Ak})
\begin{equation}\label{eq11} P_n(z)=-P_n(0)B_n(z)+Q_n(0)D_n(z),
\end{equation}
so by the Cauchy-Schwarz inequality
\begin{equation}\label{eq:CS}
|P_n(z)|^2\le (P_n^2(0)+Q_n^2(0))(|B_n(z)|^2+|D_n(z)|^2).
\end{equation}
Combined with \eqref{eq:BDalpha} we get
\begin{equation}\label{eq:Pn2alpha}
|P_n(z)|^{2\alpha}\le
(P_n^2(0)+Q_n^2(0))^\alpha\exp(2K(\alpha)|z|^\alpha),
\end{equation}
which shows that 
$\sum_{n=0}^\infty|P_n(z)|^{2\alpha}$
converges uniformly on compact subsets of $\mathbb C$.

Similarly we have
$$
Q_n(z)=-P_n(0)A_n(z)+Q_n(0)C_n(z),
$$
leading to the estimate
$$
|Q_n(z)|^{2\alpha}\le
(P_n^2(0)+Q_n^2(0))^\alpha\exp(2K(\alpha)|z|^\alpha),
$$
and the assertion  $(Q_n^2(z))\in\ell^{\alpha}$.
By \eqref{eq:CS} and \eqref{eq:BDalpha} we also get
\begin{multline}\label{eq:sz} 
P^2(z)=\sum_{n=0}^\infty |P_n(z)|^2\le \sum_{n=0}^\infty (P_n^2(0)+Q_n^2(0))(|B_n(z)|^2+|D_n(z)|^2) \\
\le \left(\sum_{n=0}^\infty (P_n^2(0)+Q_n^2(0))\right) \exp(2\alpha^{-1}K(\alpha)|z|^\alpha),
\end{multline}
showing \eqref{eq:Porder}, from which we clearly get that
$\rho=\rho_P\le\alpha$, and if $\rho=\alpha$, then
 $\tau=\tau_P\le  K$.

From the well-known formula
\begin{equation}\label{eq:LO}
P_{n-1}(z)Q_n(z)-P_n(z)Q_{n-1}(z)=\frac{1}{b_{n-1}},
\end{equation}
cf. \cite[p. 9]{Ak}, we get
\begin{equation}\label{eq:LO1}
\frac{2}{b_{n-1}}\le |P_{n-1}(z)|^2 + |P_{n}(z)|^2 + |Q_{n-1}(z)|^2 + |Q_{n}(z)|^2,
\end{equation}
hence
$$
\frac{2^\alpha}{b_{n-1}^\alpha}\le |P_{n-1}(z)|^{2\alpha} + |P_{n}(z)|^{2\alpha} + |Q_{n-1}(z)|^{2\alpha} + |Q_{n}(z)|^{2\alpha},
$$
which shows that $(1/b_n)\in\ell^\alpha$.
\end{proof}

We next give an almost converse theorem to Theorem~\ref{thm:order1}, under the
Berezanski\u{\i} assumptions. It is a slight sharpening of Theorem~\ref{thm:beralphaintro} because we have replaced \eqref{eq:origber} by 
\eqref{eq:finalber}.
 
\begin{thm}\label{thm:beralpha} Assume that the coefficients of \eqref{eq:3t}  satisfy 
$$
\sum_{n=1}^\infty \frac{1+|a_n|}{\sqrt{b_nb_{n-1}}}<\infty,
$$
and that either \eqref{eq:logconvex} or \eqref{eq:logconcave} holds.
Assume in addition that $P$ satisfies 
$$
P(z)\le C\exp(K|z|^\alpha)
$$
for some $\alpha$ such that $0<\alpha<1$ and suitable constants $C,K>0$. 
 
Then
\begin{equation}\label{eq:Pnalpha}
1/b_n, P_n^2(0),Q_n^2(0)=O(n^{-1/\alpha}),
\end{equation}
so in particular $(1/b_n),(P_n^2(0)),(Q_n^2(0))\in \ell^{\alpha+\varepsilon}$ for any $\varepsilon>0$.
\end{thm}

\begin{proof} Using that $b_{n-1} < b_n$ for $n\ge n_1$, we get
  $b:=\min\{b_k\}>0$. For $n\ge n_1$ we find
\begin{equation}\label{eq:consBer}
\frac{1}{b_{n-1}^{2n}}\le \frac{1}{b^{2n_1}b_{n-1}^{2(n-n_1)}}\le
Ab_{n,n}^2\le Ac_n^2,
\end{equation}
where we have used \eqref{eq:lead},  \eqref{eq:cn} and
$$
A=\left(\frac{b_0\cdots b_{n_1-1}}{b^{n_1}}\right)^2.
$$
Next, \eqref{eq:ckP} leads to
$$
\sum_{n=n_1}^\infty   \left (\frac{r}{b_{n-1}}\right )^{2n}\le A\sum_{n=0}^\infty c_n^2 r^{2n}
= \frac{A}{2\pi}\int_0^{2\pi}P^2(re^{it})\,dt \le
AC^2\exp[2Kr^\alpha].
$$
Therefore, for any $n\ge n_1,r>0$
\begin{equation}\label{calc}
 \frac{r}{b_{n-1}}\le (AC^2)^{1/2n}\exp[Kr^\alpha/n].
\end{equation}
For $r=n^{1/ \alpha}$ we obtain 
$$
\frac{1}{b_{n-1}}=O(n^{-1/\alpha}),\;n\to\infty.
$$
Now in view of \eqref{eq:spec} we get \eqref{eq:Pnalpha}. 
\end{proof}

\begin{defn}\label{thm:convexpo} {\rm For a sequence  $(z_n)$ of
    complex numbers for which $|z_n|\to\infty$, we introduce the {\it exponent of
      convergence}
$$
\mathcal E(z_n)=\inf \left\{\alpha>0 \mid \sum_{n=n^*}^\infty
\frac{1}{|z_n|^\alpha}<\infty \right\},
$$
where $n^*\in\mathbb N$ is such that $|z_n|>0$ for $n\ge n^*$. 

The counting function of $(z_n)$ is defined as
$$
n(r)=\#\{n \mid |z_n|\le r\}.
$$
}
\end{defn}

The following result is well-known, cf. \cite{BOAS},\cite{Le}.

\begin{lem}\label{thm:expc=or} 
$$
\mathcal E(z_n)=\limsup_{r\to\infty}\frac{\log n(r)}{\log r}.
$$
\end{lem}

\begin{thm}\label{thm:beralpha1} Assume that the coefficients of \eqref{eq:3t}  satisfy 
$$
\sum_{n=1}^\infty \frac{1+|a_n|}{\sqrt{b_nb_{n-1}}}<\infty,
$$
and that either \eqref{eq:logconvex} or \eqref{eq:logconcave} holds. 

Then the order  $\rho$ of the moment problem is given by $\rho=\mathcal
E(b_n)$.
\end{thm}

\begin{proof} We first show that $\mathcal E(b_n)\le\rho_P$. This is
  clear if $\rho_P=1$ because 
by assumption $\mathcal E(b_n)\le 1$. If
  $\rho_P<1$ then $P$ satisfies 
$$
M_P(r)\le_{as}\exp(r^\alpha)
$$
for any $\alpha>\rho_P$. By \eqref{eq:Pnalpha} we then have $\sum
1/b_n^{\alpha+\varepsilon}<\infty$ for $\alpha>\rho_P$ and
$\varepsilon>0$, hence $\mathcal E(b_n)\le \rho_P$.

By \eqref{eq:berez} we get for $r_n=P_n$ 
\begin{equation}\label{eq:Pbn}
    P(z)\le c\;\left(\sum_{n=0}^\infty\frac{1}{b_{n-1}}\right)^{1/2}\;\Pi(|z|),  
\end{equation}
and the infinite product $\Pi(z)$
is an entire function of order equal to
$\mathcal  E(b_n)$ by Borel's Theorem, cf. \cite{Le}, hence $\rho_P\le \mathcal E(b_n)$. 
\end{proof}

\begin{ex}\label{thm:ex0} {\rm For $\alpha>1$ let $b_n=(n+1)^\alpha,a_n=0,n\ge 0$. The three-term recurrence relation \eqref{eq:3t} with these coefficients determine the orthonormal polynomials of a symmetric indeterminate moment problem satisfying \eqref{eq:conv} and \eqref{eq:logconcave}.
 By Theorem~\ref{thm:beralpha1} the order of the moment problem is $1/\alpha$.

Similarly, $b_n=(n+1)\log^\alpha(n+2),a_n=0$ lead for $\alpha>1$ to a
symmetric indeterminate moment problem of order 1 and type 0.} 
\end{ex}

Theorem~\ref{thm:order1} and Theorem~\ref{thm:beralpha} can be generalized in order to capture much slower
types of growth of the moment problem, as well as   growth  faster than any
order. This is done in the following section.

\section{Order functions}\label{sec-ord}

\begin{defn}\label{thm:alphaorder}
{\rm By an {\bf order function}\footnote{There is no direct relation
    between this concept and Valiron's concept of a proximate order
    studied in \cite{Le}.}
we understand a continuous, positive and
  increasing function $\alpha:(r_0,\infty)\to \mathbb R$ with
  $\lim_{r\to\infty}\alpha(r)=\infty$ and such that the
  function $r/\alpha(r)$ is also increasing with $\lim_{r\to\infty}r/\alpha(r)=\infty$. Here $0\le r_0<\infty$.}
\end{defn}

If $\alpha$ is an order function , then so is $r/\alpha(r)$.

\begin{defn}\label{thm:dualorder} {\rm For an order function $\alpha$ as
  above, the function
$$
\beta(r)=\frac{1}{\alpha(r^{-1})},\quad 0<r<r_0^{-1}
$$
 will be called the {\bf dual function}.  Since $\lim_{r\to
   0}\beta(r)=0$, we define $\beta(0)=0$. Note that $\beta$ as well as $r/\beta(r)$
are increasing.}
\end{defn}

 Observe that the dual function satisfies
\begin{equation}\label{eq:dualord1}
\beta(Kr)\le K\beta(r),\quad
K>1,\; 0<Kr<r_0^{-1},
\end{equation}
\begin{equation}\label{eq:dualord2}
\beta (r_1+r_2)\le \beta (2\max(r_1,r_2))\le 2\beta(\max(r_1,r_2))\le 
2\beta(r_1)+2\beta (r_2),
\end{equation}
for $2\max(r_1,r_2)<1/r_0$.
\medskip

\begin{ex}\label{thm:ex1} {\rm Order functions.} 

 {\rm {\bf 1.}  The function $\alpha(r)=r^\alpha $ with $0<\alpha<1$ 
  satisfies the assumptions of an order function with $r_0=0$, 
 and $\beta(r)=\alpha(r).$

\medskip
{\bf 2.} The function $\alpha(r)=\log ^\alpha r$ with $\alpha>0$
  satisfies the assumptions of an order function with $r_0=\exp(\alpha)$ and 
$$
\beta(r)=\frac{1}{(-\log r)^\alpha }.
$$

\medskip
{\bf 3.} The function $\alpha(r)=\log^{\alpha}\log r$ with $\alpha>0$
  is an order function with $r_0>\rm{e}$ being the unique solution to 
$(\log r)\log\log r=\alpha$.

\medskip
{\bf 4.} If $\alpha$ is an order function, the so are $c\alpha(r)$ and $\alpha(cr)$ for $c>0$.

\medskip
{\bf 5.} If $\alpha_1$ and $\alpha_2$ are order functions, then also $\alpha_1(\alpha_2(r))$ is an order function for $r$ sufficiently large.

\medskip
{\bf 6.} The function $\alpha(r)=(\log^{\alpha} r)\log^{\beta}\log r$ is an order function for any $\alpha,\beta>0$,
because
$$
\frac{r}{\alpha(r)}=\left[\frac{r^{1/(\alpha+\beta)}}{(\alpha+\beta)\log r^{1/(\alpha+\beta)}}\right]^{\alpha+\beta}\left[\frac{\log r}{\log\log r}\right]^\beta
$$
shows that  $r/\alpha(r)$ is increasing for $r>r_0:=\exp(\max(e,\alpha+\beta))$.
}
\end{ex}

\begin{defn}\label{thm:def1} Let $\alpha$ be an order function.
A continuous unbounded function $f:\mathbb C\to\mathbb C$ is said to have order bounded by $\alpha(r)$ if
$$
 M_f(r)\le_{as} e^{K\alpha(r)\log r}=r^{K\alpha(r)},
$$
 for some constant $K.$
 \end{defn}

For $f$ as above to have order bounded by $\alpha(r)=\log^{\alpha}r$
for some $\alpha>0$, is the same as to have finite logarithmic order in the sense of
Section~\ref{sec-pre}.

Given an order function $\alpha:(r_0,\infty) \to \mathbb R$ and its
dual $\beta$, we are in the following going to consider expressions
$\beta(u_n)$, where $\{u_n\}$ is a sequence of non-negative numbers
tending to zero. This means that $\beta(u_n)$ is only defined for $n$
sufficiently large, so assertions like
$$
\sum_{n}^\infty \beta(u_n)<\infty,\quad \beta(u_n)=O(1/n)
$$ 
make sense. The first assertion means that 
$$
\sum_{n=N}^\infty \beta(u_n)<\infty
$$
for one $N$ (and then for all $N$) so large that $\beta(u_n)$ is defined for $n\ge N$.

We begin   by proving two lemmas.

\begin{lem}\label{thm:counting} Let $\alpha:(r_0,\infty)\to(0,\infty)$
  be an order function
  with dual function $\beta$ and let $\{u_n\}_{n=1}^\infty$ be a
  sequence of positive numbers such that $u_n\to 0$
and $u_n<1/r_0$ for all $n\ge n_0$.
 
For any  number $r >0$ let $A_r=\{ n \mid  u_n\ge r^{-1}\}$ and $N_r=\# A_r$.  

\begin{enumerate}
\item[(a)] Assume $\sum_{n}^\infty \beta(u_n)<\infty. $ Then $ N_r=O(\alpha (r))$.
\item[(b)] Assume $N_r=O(\alpha(r))$. Then for any $\varepsilon>0$
$$
\sum_{n}^\infty \beta^{1+\varepsilon}(u_n)<\infty.
 $$
\end{enumerate}
\end{lem}

\begin{proof}
Let $v_n$ be the decreasing rearrangement of the sequence $u_n.$ Then 
$$
 N_r=\#\{ n\,\mid \, v_n\ge r^{-1}\},
$$
and since $\beta(r)$ is increasing, we find for $r>r_0$
$$
N_r\le n_0-1+ \#\{ n\ge n_0\,\mid\, \beta(v_n)\ge \beta(r^{-1})\}.
$$
(a) We have $\sum_{n}^{\infty} \beta(v_n)<\infty$, hence $n\beta
(v_n)\to 0$ and thus $n\beta(v_n)\le K$ for $n\ge n_0$ and a suitable constant $K.$
Furthermore, 
\begin{eqnarray*} N_r &\le& n_0-1+\#\left \{ n\ge n_0\,\mid\, \frac{K}{n}\ge \beta(r^{-1})\right \}\\
&=& n_0-1+\#\left \{ n\ge n_0\,\mid\,  n\le K\alpha(r) \right \},
\end{eqnarray*}
showing that $N_r=O(\alpha(r))$.

\noindent (b) Assume $N_r=O(\alpha(r)).$ Observing that
$N_{v^{-1}_n}\ge n$ we get
$n\le K\alpha(v^{-1}_n),$ for $n$ sufficiently large and suitable $K$, i.e.,
 $\beta(v_n)=O(1/n)$, which implies the conclusion.
\end{proof}

 \begin{lem}\label{thm:ineq} Assume the conditions of
   Lemma~\ref{thm:counting}(a). For $r>r_0$ we then have
 $$
\log\prod_{n=1}^\infty (1+ru_n)\le N_r[\log r +C]
+\alpha(r)\sum_{n\notin A_{r_0}}\beta(u_n),
$$
 where  $C=\max \{\log (2u_n)\}.$
 \end{lem}
 \begin{proof}  For $n\in A_r$ we have $ru_n\ge 1$, hence
 $$
\log(1+ru_n)\le \log 2ru_n=\log r +\log (2u_n)\le \log r +C.
$$  
 Furthermore,  for $r>r_0,n\notin A_r$ we have $u_n<r^{-1}$, and using
 that $s/\beta(s)$ is increasing leads to
 $$
 ru_n=\frac{u_n}{r^{-1}}\le
 \frac{\beta(u_n)}{\beta(r^{-1})}=\alpha(r)\beta (u_n).
$$
 Thus, for $r>r_0$
 \begin{multline*}
\log \prod_{n=1}^\infty (1+ru_n)=\sum_{n\in A_r}\log
 (1+ru_n)+\sum_{n\notin A_r}\log (1+ru_n)\\
 \le N_r [\log r +C]+\sum_{n\notin A_r} \alpha(r)\beta(u_n)\le 
 N_r [\log r + C] + \alpha(r)\sum_{n\notin A_{r_0}} \beta(u_n).
 \end{multline*}
 \end{proof}
 Combining Lemma~\ref{thm:counting}(a) and Lemma~\ref{thm:ineq} gives immediately the following.

 \begin{prop}\label{thm:ineq1}  Let $\alpha:(r_0,\infty)\to(0,\infty)$
  be an order function
  with dual function $\beta$, and let $\{u_n\}_{n=1}^\infty$ be a
  sequence of positive numbers such that $u_n\to 0$
and $u_n<1/r_0$ for all $n\ge n_0$. Under the assumption
$\sum_n^\infty \beta(u_n)<\infty$
$$
\log\prod_{n=1}^\infty (1+ru_n)=O(\alpha(r)\log r),
$$
and in particular  the entire function
$$
f(z)=\prod_{n=1}^\infty(1+zu_n)
$$
has order bounded by $\alpha$.
 \end{prop}

Theorem \ref{thm:order1} and  \ref{thm:beralpha}
can be considered as results about the order function $\alpha(r)=r^\alpha,0<\alpha<1$.

Theorem \ref{thm:beta1} and  \ref{thm:order2} below are similar
results for arbitrary order functions.
 The price for the generality is an extra log-factor, so the
 generalization is mainly of interest
 for orders of slower growth than $\alpha(r)=r^\alpha$. For the order
 $\alpha(r)=r^\alpha$ it is better
 to refer directly to the results of Section~\ref{sec-ber}.

\begin{thm}\label{thm:beta1} For an order function $\alpha$ with dual function $\beta$ the
  following conditions are equivalent for a given indeterminate moment problem:
\begin{enumerate}
\item[(i)] $\beta(P_n^2(0)),\beta(Q_n^2(0))\in\ell^1$,
\item[(ii)] $\beta(|P_n(z)|^2),\beta(|Q_n(z)|^2)\in\ell^1$ for all
  $z\in\mathbb C$.
\end{enumerate}
If the conditions are satisfied, then the two series indicated in (ii) converge
uniformly on compact subsets of $\mathbb C$.

Furthermore, $\beta(1/b_n)\in \ell^1$ and $P$ has order bounded by $\alpha$.
\end{thm}

\begin{proof}  Condition (ii) is clearly stronger than condition (i). 

Assume next that (i) holds. 
By \eqref{eq:Pn2alpha} for $\alpha=1$
\begin{equation}\label{eq:zto0}
|P_n(z)|^2\le (P_n^2(0)+Q_n^2(0))\exp(2K(1)|z|),
\end{equation}
so by \eqref{eq:dualord1} and  \eqref{eq:dualord2} we get for $n$
sufficiently large
\begin{equation}\label{eq:betaz0}
\beta(|P_n(z)|^2)\le 2\exp(2K(1)|z|)\left(\beta(P_n^2(0))+ \beta(Q_n^2(0))\right).
\end{equation}
This shows that $\sum\beta(|P_n(z)|^2)$ converges uniformly on compact
subsets of $\mathbb C$. 

The assertion $\beta(|Q_n(z)|^2)\in\ell^1$ is proved similarly.

 By (\ref{simon})   and Proposition~\ref{thm:ineq1} 
  we obtain 
\begin{equation}\label{eq:BnDnlog}
\sqrt{|B_n(z)|^2+|D_n(z)|^2}\le \exp(L \alpha(|z|)\log |z|),
\end{equation}
 for some constant $L$ and $|z|$ sufficiently large. Using \eqref{eq:CS} and \eqref{eq:Kalpha} (with $\alpha=1$)
 we then get for large $|z|$
 $$
 P^2(z)=\sum_{n=0}^\infty |P_n(z)|^2 \le K(1) \exp(2L \alpha(|z|)\log
 |z|),
$$
which shows that $P$ has order bounded by $\alpha$.

From the inequality \eqref{eq:LO1} we immediately get that $\beta(1/b_n)\in\ell^1$.
\end{proof}

\begin{thm}\label{thm:order2}  Assume that  the coefficients of \eqref{eq:3t}  satisfy 
$$
 \sum_{n=1}^\infty \frac{1+|a_n|}{\sqrt{b_nb_{n-1}}}<\infty,
$$
and that either \eqref{eq:logconvex} or \eqref{eq:logconcave} holds.
Assume in addition that the function $P(z)$ has order bounded by some given order function $\alpha$.  

\begin{enumerate}
\item[(i)] If there is $0<\alpha<1$ so that $r^\alpha \le_{\text{\tiny as}}\alpha(r)$,
then 
$$ \beta(1/b_n), \beta(P^2_n(0)), \beta(Q_n^2(0))=O\left(\frac{\log n}{n}\right).
$$

\item[(ii)] If $\alpha(r^2)= O(\alpha(r))$,
then
$$  \beta(1/b_n), \beta(P^2_n(0)), \beta(Q_n^2(0))=O(1/n).
$$
\end{enumerate}

 In both cases
$$
\beta(1/b_n), \beta(P_n^2(0)),
\beta(Q_n^2(0))\in\ell^{1+\varepsilon}
$$
 for any $\varepsilon>0.$
\end{thm}

\begin{proof}
Inserting the estimate
$$
M_P(r)\le_{\text{\tiny as}}\exp(K\alpha(r)\log r)
$$
in \eqref{eq:ckP}, we get
$$
\sum_{k=0}^\infty r^{2k}c_k^2 \le_{\text{\tiny as}}\exp(2K\alpha(r)\log r),
$$
hence by \eqref{eq:consBer}
\begin{equation}\label{eq:help1}
\sum_{n=n_1}^\infty\left(\frac{r}{b_{n-1}}\right)^{2n}\le_{\text{\tiny as}}A\exp(2K\alpha(r)\log r).
\end{equation}
 Choose $r_1>\max(1,r_0)$ so large that the inequality in
\eqref{eq:help1} holds for $r\ge r_1$.
In particular we have
\begin{equation}\label{eq:w}
\frac{r}{b_{n-1}}\le A^{1/2n}
\exp((K/n)\alpha(r)\log r),\quad n\ge n_1,r\ge r_1.
\end{equation} 

Consider (i). For any $n>K\alpha(r_1)\log r_1$ it is
possible by continuity of $\alpha$ to choose $r=r_n>r_1$ such that
 \begin{equation}\label{eq:c}
K\alpha(r_n)\log r_n=n.
\end{equation}  
For sufficiently large $n$ we then have
$$
\frac{1}{b_{n-1}}\le \frac{A^{1/(2n)}e}{r_n} <\frac{3}{r_n}.
$$

Since $\beta$ is increasing, we get for sufficiently large $n$
by \eqref{eq:dualord1}  and \eqref{eq:c}
\begin{equation}\label{eq:d}
   \beta(1/b_{n-1})\le\beta(3/r_n)\le 3\beta(1/r_n)=\frac{3}{\alpha(r_n)}=
\frac{3K\log r_n}{n}.
\end{equation}
But \eqref{eq:c} and the assumption  $r^\alpha\le_{\text{\tiny as}}\alpha(r)$ imply
that $Kr_n^\alpha\log r_1 \le n,$ for large $n.$ Thus
$\log r_n=O(\log n)$, and by \eqref{eq:d} we get
$$
\beta(1/b_{n-1})= O(\frac{\log n}{n}).
$$
In view of \eqref{eq:spec}  we get that
$\beta(P_n^2(0)),\beta(Q_n^2(0))=O(\log n/n)$.

\medskip

We turn now to the case (ii), where $\alpha(r^2)=O(\alpha(r))$. For
any $n>2K\alpha(r_1)$ we now choose $r_n$ such that
\begin{equation}\label{eq:c1}
K\alpha(r_n)=\frac{n}{2}.
\end{equation} 
Then \eqref{eq:w} yields
$$
\frac{1}{b_{n-1}}\le \frac{A^{1/2n}}{\sqrt{r_n}}<\frac{2}{\sqrt{r_n}}
$$
for $n$ sufficiently large.
Thus
$$
\beta(1/b_{n-1})\le \beta (2/\sqrt{r_n})\le 2\beta(1/\sqrt{r_n})=
\frac{2}{\alpha(\sqrt{r_n})}.
$$
By assumption there exists $d>0$ such that $\alpha(\sqrt{r_n})\ge d \alpha(r_n)$ for $n$ large enough.
Thus in view of \eqref{eq:c1} we find
$$
\beta(1/b_{n-1})\le \frac{2}{d\,\alpha(r_n)}=\frac{4K}{d\,n}.
$$
As above, the conclusion follows from \eqref{eq:spec}.
\end{proof}

\begin{rem}\label{thm:rem41} {\rm The following order functions
    satisfy the assumption (i) of Theorem~\ref{thm:order2}:
$$
\alpha(r)=r^\alpha, \;0<\alpha<1,\qquad
\alpha(r)=\frac{r}{\log^\alpha r},\;\alpha>0.
 $$
On the other hand the functions
 $$
\alpha(r)=\log^\alpha r,\quad \alpha(r)=\log^\alpha\log r,\quad \alpha(r)=(\log^{\alpha} r)\log^{\beta} \log r,\;\alpha,\beta>0
$$
satisfy (ii).

Although $\alpha(r)=r/\log^\alpha r$ is an order function for any $\alpha>0$, then an entire function $f$ of order bounded by $\alpha(r)$
is only of minimal exponential type under the assumption $\alpha>1$.
}
\end{rem} 

\begin{ex} {\rm Consider a moment problem of logarithmic order $\rho^{[1]}$ satisfying $0<\rho^{[1]}<\infty$ and of finite logarithmic type $\tau^{[1]}$.
Assume that
$a_n,b_n$ satisfy the conditions of  Theorem~\ref{thm:order2}. Then $P$ has order bounded by the order $\alpha(r)=(\log r)^{\rho^{[1]}}$. Since the
case (ii) occurs, and since $\beta(r)=\log^{-\rho^{[1]}}(1/r)$, we have
$$
\log^{-\rho^{[1]}}(b_n), \log^{-\rho^{[1]}}(P_n^{-2}(0)), \log^{-\rho^{[1]}}(Q_n^{-2}(0)) =O(1/n).
$$
Therefore
$$
1/b_n, P_n^2(0), Q_n^2(0)=O(e^{-Cn^{1/\rho^{[1]}}})
$$
for a suitable constant $C>0$.
From \eqref{eq:zto0}  we also get 
$$
|P_n^{2}(z)|=O(e^{-Cn^{1/\rho^{[1]}}}),
$$
uniformly on compact subsets of $\mathbb C$. These results can be applied to Discrete $q$-Hermite II polynomials, where
$a_n=0,\;b_n=q^{-n-1/2}(1-q^{n+1})^{1/2}$, cf. \cite{K:S}, and to $q^{-1}$-Hermite polynomials, where $a_n=0,\; b_n=(1/2)q^{-(n+1)/2}(1-q^{n+1})^{1/2}$, 
 cf. \cite{I:M}. In both cases $0<q<1$ and $(b_n)$ is log-concave, $\w{}=1$.
}
\end{ex}

In analogy with Theorem~\ref{thm:beralpha1} the logarithmic order of an indeterminate
moment problem of order zero can be determined by the growth of
$(b_n)$, provided the Berezanski\u{\i} conditions hold.

\begin{thm}\label{thm:Berlogorder} Assume that the coefficients of \eqref{eq:3t}  satisfy 
$$
 \sum_{n=1}^\infty \frac{1+|a_n|}{\sqrt{b_nb_{n-1}}}<\infty
$$
and that either  \eqref{eq:logconvex} or \eqref{eq:logconcave} holds. Assume further that the moment problem has order 0.

Then the logarithmic order $\rho^{[1]}$ of the moment problem is
given as $\rho^{[1]}=\mathcal E(\log b_n)$.
\end{thm}

\begin{proof} We first establish that $\rho^{[1]}\ge \mathcal
  E(\log b_n)$, which is clear if $\rho^{[1]}=\infty$. If
  $\rho^{[1]}<\infty$ we know that for every $\varepsilon>0$
$$
M_P(r)\le_{\text{\tiny as}}r^{(\log r)^{\rho^{[1]}+\varepsilon}}.
$$
In other words $P$ has order bounded by $\alpha(r)=(\log
r)^{\rho^{[1]}+\varepsilon}$, so by
Theorem~\ref{thm:order2}(ii) we know that 
$$
\beta(1/b_n)=\frac{1}{(\log b_n)^{\rho^{[1]}+\varepsilon}}\in\ell^{1+\varepsilon},
$$
hence $\mathcal
E(\log b_n)\le(\rho^{[1]}+\varepsilon)(1+\varepsilon)$ for any
$\varepsilon>0$, thus $\mathcal E(\log b_n)\le \rho^{[1]}$.

From \eqref{eq:Pbn} we get $\w{P}\le \w{\Pi}$. 
However,  $\w{\Pi}= \mathcal E(\log b_n)$ by Proposition~5.4
 in \cite{B:P:H}.
\end{proof}

\begin{ex}\label{thm:exlogorderalpha}
{\rm For $a>1,\alpha>0$  let $b_n=a^{n^{1/\alpha}}$, and let
    $|a_n|\le a^{c n^{1/\alpha}}$ for some $0<c<1$.
The three-term recurrence relation \eqref{eq:3t} with these
coefficients determine orthogonal polynomials of an indeterminate
moment problem satisfying \eqref{eq:conv} and \eqref{eq:logconvex} or
\eqref{eq:logconcave} according to
$$
b_n^2\;\left\{\begin{array}{lll} = \\ < \\
    >\end{array}\right\}\;b_{n-1}b_{n+1}\;\Leftrightarrow\;\left\{
\begin{array}{lll} \alpha=1 \\ \alpha <1 \\ \alpha>1\end{array}\right..
$$
We find
$\mathcal E(b_n)=0$  and $\mathcal E(\log b_n)=\alpha$,  so by
Theorem~\ref{thm:beralpha1} and Theorem~\ref{thm:Berlogorder} the moment problem has order 0 and logarithmic
order $\rho^{[1]}=\alpha$.
}
\end{ex}

\begin{ex}\label{thm:exlogtype} {\rm
For $a>1$ and $\alpha >0$ consider the product
$$
f(r)=\prod_{n=1}^\infty \left (1+\frac{r}{a^{n^{1/\alpha}}}\right )
$$
appearing in Lemma~\ref{thm:ineq} with $u_n=a^{-n^{1/\alpha}}$.
Let 
$$
\alpha(r)=(\log^\alpha r) \,(\log \log r)^2
$$
be an order function of the type considered in Example~\ref{thm:ex1} (6). We can use $r_0=\exp(\max(e,2+\alpha))$ and $u_n<1/r_0$ for $n>n_0$ with
$$
n_0=\left(\frac{\max(e,2+\alpha)}{\log a}\right)^{\alpha}.
$$

For $N_r=\#\{n\,|\, a^{n^{1/\alpha}}\le r  \}$ we have
\begin{equation}\label{eq:Rys1}
\left (\frac{\log r}{\log a}\right
  )^{\alpha}-1<N_r\le \left (\frac{\log r}{\log a}\right
  )^{\alpha}.
\end{equation}
 Moreover,
$$
\beta(u_n)=\frac{1}{\alpha(u_n^{-1})}=\frac{1}{(\log
  a)^{\alpha}}\,\, \frac{1}{n\,[(1/\alpha)\log n+\log\log a]^{2}}
 $$
satisfies 
$$
C:=\sum_{n>(1/\log a)^\alpha}^\infty \beta(u_n)<\infty.
$$
The proof of Lemma~\ref{thm:ineq} gives
$$
\log f(r)\le \sum_{n=1}^{N_r} \log\left ( 2\frac{r}{a^{n^{1/\alpha}}}\right )+
C\alpha(r)=\sum_{n=1}^{N_r} \log\left ( \frac{r}{a^{n^{1/\alpha}}}\right )+N_r\log 2+C\alpha(r).
 $$

On the other hand
$$ \log f(r) \ge \sum_{n=1}^{N_r} \log\left (1+\frac{r}{a^{n^{1/\alpha}}}\right )\ge 
\sum_{n=1}^{N_r}\log \left (\frac{r}{a^{n^{1/\alpha}}}\right ).
$$
We have
$$
  \sum_{n=1}^{N_r}\log \left (\frac{r}{ a^{n^{1/\alpha}}}\right ) =
N_r\log r- \log a\sum_{n=1}^{N_r} {n^{1/\alpha}}  
  $$
  and
  $$
\frac{1}{1+1/\alpha}N_r^{1+1/\alpha}\le \sum_{n=1}^{N_r} n^{1/\alpha} \le \frac{1}{1+1/\alpha}(N_r+1)^{1+1/\alpha}.
$$
  Therefore, in view of \eqref{eq:Rys1} we get
   $$
\log f(r) = \frac{1}{(\alpha +1)(\log a)^{\alpha}}\,(\log
   r)^{1+\alpha}\,[1+o(1)],
 $$
showing that the logarithmic order is $\alpha$ (as we already know
from Example~\ref{thm:exlogorderalpha}), and the logarithmic type is 
$$
\frac{1}{(\alpha+1)(\log a)^{\alpha}}.
$$
}
\end{ex}
 
 \begin{ex}\label{thm:exdlo} {\rm For $a,b>1$ let $b_n=a^{b^n}$ and $|a_n|\le a^{cb^n}$
     with $bc<1$. In this case $(b_n)$ is logarithmic convex, and the
     coefficients lead to an indeterminate moment problem with order
     as well as logarithmic order equal to 0.

This motivates a study of functions bounded by the order function
$\alpha(r)=(\log\log r)^\alpha$, considered in the next section.}
\end{ex}

\section{Double logarithmic order}\label{sec-double}

For an unbounded continuous function $f$  we define the {\em
 double  logarithmic order}  $\ww{f}$ as
$$
\ww{f} = \inf \{ \alpha >0\, | \, M_f(r)\le_{\text{\tiny as}}r^{(\log\log r)^{\alpha}}\,
\},
$$
where $\ww{f}=\infty$, if there are no $\alpha>0$ satisfying the asymptotic inequality. Of course $\ww{f}<\infty$ is only possible if $\w{f}=0$.

In case $0<\ww{}=\ww{f}<\infty$ we define the {\em double logarithmic type} as
$$
\dd{f}=\inf \{ c >0\, | \, M_f(r)\le_{\text{\tiny as}} r^{c(\log\log r)^{\ww{}}}\,
\}.
$$

\begin{thm}\label{thm:dlog} For an indeterminate moment problem of
  logarithmic order zero the functions $A,B,C,D,P,Q$ have the same
  double logarithmic order $\ww{}$ and type $\dd{}$ called the
  double logarithmic order and type of the moment problem.
\end{thm}

The proof of this result can be done exactly in the same way as the corresponding proof for logarithmic order and type in \cite{B:P:H}, so 
we leave the details to the reader.

For an entire transcendental function $f(z)=\sum_{n=0}^\infty a_nz^n$ of logarithmic
order 0 the double logarithmic order and type can be expressed in
terms of the coefficients $a_n$ by the following formulas.

\begin{thm}\label{thm:dlogcoef}
\begin{eqnarray}\label{eq:dlogorder}
\ww{f} 
= \limsup_{n\to \infty}\frac{\log n}{\log\log \log\left(
    \frac{1}{\sqrt[n]{|a_n|}}\right)},
\end{eqnarray}
and if $0<\ww{}=\ww{f}<\infty$
\begin{eqnarray}\label{eq:dlogtype}
\dd{f}=
 \limsup_{n\to
  \infty}\frac{n}{\left(\log\log \frac{1}{\sqrt[n]{|a_n|}}\right)^{\ww{}}}.
\end{eqnarray}
\end{thm}

The proof is given in the Appendix.

The results of Section~\ref{sec-phi} about $\Phi$ can also be generalized:

\begin{thm}\label{thm:dord2} Suppose the logarithmic order of the moment problem is zero. Then $\Phi$
  has the same double logarithmic order $\ww{}$ and type $\dd{}$ as the moment problem.
\end{thm}

\begin{proof} From the inequality $M_D(r)\le c_0rM_{\Phi}(r)$, cf. \eqref{eq:D-Phi}, we get $\ww{}=\ww{D}\le \ww{\Phi}$.
For any $\varepsilon>0$ we have 
$$
P(re^{i\theta})\le r^{(\log\log r)^{\ww{}+\varepsilon}}
$$
for $r$ sufficiently large, which by \eqref{eq:ckP} leads to
$\ww{\Psi}\le\ww{}$, where $\Psi$ is given by \eqref{eq:Psi}. From Theorem~\ref{thm:dlogcoef} we see that
$\ww{\Phi}=\ww{\Psi}$ and hence $\ww{}=\ww{\Phi}$. The proof concerning type follows using similar ideas. 
\end{proof}

\begin{thm}\label{thm:Berdlogorder} Assume that the coefficients of \eqref{eq:3t}  satisfy 
$$
 \sum_{n=1}^\infty \frac{1+|a_n|}{\sqrt{b_nb_{n-1}}}<\infty
$$
and that either  \eqref{eq:logconvex} or \eqref{eq:logconcave} holds.

Then the double logarithmic order $\ww{}$ of the moment problem is
given as $\ww{}=\mathcal E(\log\log b_n)$.
\end{thm}

\begin{proof} We first establish that $\ww{}\ge \mathcal
  E(\log\log b_n)$, which is clear if $\ww{}=\infty$. If
  $\ww{}<\infty$ we know that for every $\varepsilon>0$
$$
M_P(r)\le_{\text{\tiny as}}r^{(\log\log r)^{\ww{}+\varepsilon}}.
$$
In other words $P$ has order bounded by $\alpha(r)=(\log\log
r)^{\ww{}+\varepsilon}$, so by
Theorem~\ref{thm:order2}(ii) we know that 
$$
\beta(1/b_n)=\frac{1}{(\log\log b_n)^{\ww{}+\varepsilon}}\in\ell^{1+\varepsilon},
$$
hence $\mathcal
E(\log\log b_n)\le(\ww{}+\varepsilon)(1+\varepsilon)$ for any
$\varepsilon>0$, thus $\mathcal E(\log\log b_n)\le \ww{}$.

From \eqref{eq:Pbn} we get $\ww{P}\le\ww{\Pi}$,  hence 
$\ww{}=\ww{P} = \mathcal E(\log\log b_n)$, if we prove that 
$\ww{\Pi}\le \mathcal E(\log\log b_n)$.
This is a consequence of Theorem~\ref{thm:dlc1}, but follows directly in the following way: It is clear if
$\mathcal E(\log\log b_n)=\infty$. If $\rho=\mathcal
E(\log\log b_n)<\infty$ we use Proposition~\ref{thm:ineq1} for the
order function $\alpha(r)=(\log\log r)^{\rho+\varepsilon}$ and
$u_n=1/b_n$, and since 
$$
\sum_n \beta(u_n)=\sum_n\frac{1}{(\log\log
  b_n)^{\rho+\varepsilon}}<\infty
$$
we conclude that $\log M_{\Pi}(r)=O(\alpha(r)\log r)$, hence
$\ww{\Pi}\le \rho$, because $\varepsilon>0$ can be chosen arbitrarily
small.
 
\end{proof}

\begin{ex}\label{thm:db1} {\rm
Consider
$$
f(z)=\prod_{n=1}^\infty \left(1+\frac{z}{\exp(e^{n^{1/\alpha}})}\right),
$$ 
where $0<\alpha<\infty$. We prove that $\ww{f}=\alpha,\;\dd{f}=1$. Note that $b_n=\exp(e^{n^{1/\alpha}})$
is eventually log-convex because $\exp(x^{1/\alpha})$ is convex for $x>(\alpha-1)^\alpha$ when $\alpha>1$ and convex for $x>0$ when $0<\alpha\le 1$.
This means that the indeterminate moment problem with recurrence coefficients $a_n=0$ and $b_n$ as 
above has double logarithmic order equal to $\mathcal E(\log\log b_n)=\alpha$.

Define
 $$
\alpha(r)= (\log \log r)^{2\alpha},
$$
which is an order function with $r_0=\exp(\max(e,2\alpha))$.

For $N_r=\#\{n\,|\, \exp(e^{n^{1/\alpha}})\le r  \}$ we have
\begin{equation}
 (\log\log r)^\alpha -1 < N_r \le (\log\log r)^\alpha.
\end{equation}
 Moreover, for $u_n=1/b_n$ we have
$\beta(u_n)=1/\alpha(b_n)=1/n^2$.
Observe that $\max\{\log (2u_n)\}\le 0.$ Hence Lemma~\ref{thm:ineq} gives
$$
\log f(r)\le N_r\log r +
 C\alpha(r),
 $$
where
$$
C=\sum_{n\notin A_{r_0}}^\infty \beta(u_n)\le\sum_{n=1}^\infty\frac{1}{
  n^2}.
$$
Thus
\begin{equation}\label{eq:up}
\log f(r)\le  (\log\log r)^\alpha\log r + C(\log\log r)^{2\alpha}.
\end{equation}
To minorize $\log f(r)$ we need
\begin{eqnarray*}
\sum_{n=1}^N e^{n^{1/\alpha}} &\le& e^{N^{1/\alpha}}+\int_1^N e^{x^{1/\alpha}}\,dx= e^{N^{1/\alpha}}+ \alpha\int_e^{e^{N^{1/\alpha}}} (\log t)^{\alpha-1}dt\\
&\le& \left\{\begin{array}{ll} e^{N^{1/\alpha}}(1+\alpha) \;&\mbox{for}\;\; 0<\alpha\le 1\\
 e^{N^{1/\alpha}}(1+\alpha N^{1-1/\alpha}) \;&\mbox{for}\;\; 1<\alpha.
\end{array}\right.
\end{eqnarray*}
This gives
\begin{eqnarray*}\log f(r) &\ge& \sum_{n=1}^{N_r} \log\left (1+\frac{r}{\exp(e^{n^{1/\alpha}})}\right ) \ge 
N_r\log r -\sum_{n=1}^{N_r} e^{n^{1/\alpha}}\\
&\ge&  \left\{\begin{array}{ll} \log r\left((\log\log r)^{\alpha}-2-\alpha\right)  \;&\mbox{for}\;\; 0<\alpha\le 1\\
\log r\left((\log\log r)^{\alpha}-2-\alpha(\log\log r)^{\alpha-1}\right) \;&\mbox{for}\;\; 1<\alpha.
\end{array}\right.
\end{eqnarray*}
These inequalities together with \eqref{eq:up} leads to
$$
\lim_{r\to\infty}\frac{\log f(r)}{(\log\log r)^\alpha\log r}=1,
$$ 
showing the assertion about double logarithmic order and type of $f$.}
\end{ex}

 \section{Liv\v{s}ic's function}\label{sec-lif}

For an indeterminate moment sequence $(s_n)$
 Liv\v{s}ic  \cite{L} considered the function
\begin{equation}\label{eq:lif}
F(z)=\sum_{n=0}^\infty \frac{z^{2n}}{s_{2n}}.
\end{equation}
It is entire of minimal exponential type because $\lim n/\root{2n}
\of{s_{2n}}=0$, which holds by Carleman's criterion giving that 
$$
\sum_{n=0}^\infty 1/\root 2n\of{s_{2n}}<\infty.
$$
Moreover, $\root 2n\of{s_{2n}}$ is increasing for $n\ge 1$.

Liv\v{s}ic proved  that $\rho_F\le \rho$, where $\rho$ is the order of the moment problem. It is interesting
to know whether the equality sign holds. In fact, we do not know any example with $\rho_F <\rho$. We will rather consider a
modification of Liv\v{s}ic's function given by
\begin{equation}\label{eq:lif1}
L(z)=\sum_{n=0}^\infty \frac{z^{n}}{\sqrt{s_{2n}}}.
\end{equation}
It is easy to see that $\rho_L=\rho_F$ and that $\tau_F=2\tau_L$ by
the formulas \eqref{eq:deforder} and  \eqref{eq:type}.

We shall give a new proof of the inequality $\rho_F\le \rho$ using the
function $\Phi$ from Section 2. We shall also consider the entire
function

\begin{prop}\label{thm:newLif} For an indeterminate moment problem of
  order $\rho$ we have
\begin{enumerate}
\item[(i)] $1\le s_{2n}b_{n,n}^2\le c_n^2s_{2n}$.
\item[(ii)] $M_L(r)\le M_H(r)\le M_\Phi(r),\quad r\ge 0$.
\item[(iii)] $\rho_L\le\rho_H\le\rho_\Phi=\rho$.
\item[(iv)] $\w{L}\le\w{H}\le\w{\Phi}=\w{}$, provided $\rho=0$.
\item[(v)] $\ww{L}\le\ww{H}\le\ww{\Phi}=\ww{}$, provided $\w{}=0$.
\end{enumerate}
\end{prop}

\begin{proof} By orthogonality we have
$$
1=\int P_n^2(x)\,d\mu(x)=b_{n,n}\int x^nP_n(x)\,d\mu(x),
$$
so by the Cauchy-Schwarz inequality
$$
\frac{1}{b_{n,n}}\le\left(\int x^{2n}\,d\mu(x)\right)^{1/2}\left(\int P_n^2(x)\,d\mu(x)\right)^{1/2}=\sqrt{s_{2n}},
$$
which gives the first inequality of (i). The second follows from
\eqref{eq:cn}.

The maximum modulus $M_f$ for an entire function $f(z)=\sum a_nz^n$
with $a_n\ge 0$ is given by $M_f(r)=f(r),\;r\ge 0$, and therefore (ii)
follows from (i). Finally (iii), (iv) and (v) follow from (ii).
\end{proof}

The following result gives a sufficient condition for equality in
Proposition~\ref{thm:newLif}.
 
\begin{prop}\label{thm:=inLif} If
$$
\log\root{2n}\of{c_n^2s_{2n}}=o(\log n)
$$
and in particular if
$$
c_n^2s_{2n}=O(K^n)
$$
for some $K>1$,
then $\rho=\rho_L$.

If $\rho=0$ then $\w{}=\w{L}$, and if $\w{}=0$ then $\ww{}=\ww{L}$. 
\end{prop}
 
\begin{proof} Given $\varepsilon >0$ we have for $n$ sufficiently large
\begin{equation}\label{eq:rho=rhoL}
\log\root{2n}\of{s_{2n}}\le \varepsilon\log n+\log\frac
{1}{\root{n}\of{c_n}}.
\end{equation}
Dividing by $\log n$ leads to
$$
\liminf_{n\to\infty}\frac{\log\root{2n}\of{s_{2n}}}{\log n}\le \varepsilon +
\liminf_{n\to\infty}\frac{\log\frac{1}{\root{n}\of{c_n}}}{\log n},
$$
so by \eqref{eq:deforder}
$$
\frac{1}{\rho_L}\le \varepsilon+\frac{1}{\rho},
$$
but this gives $\rho\le \rho_L$.

From \eqref{eq:rho=rhoL} we get
$$
\log\log\root{2n}\of{s_{2n}}\le \log\log\frac{1}{\root{n}\of{c_n}} +  \log\left(1+\frac{\varepsilon\log n}{\log\frac
{1}{\root{n}\of{c_n}}}\right).
$$
If $\rho=0$ the last term tends to 0, and dividing by $\log n$ we get as above $\w{}\le\w{L}$. Similarly, if $\w{}=0$ we find 
$\ww{}\le\ww{L}$.
\end{proof}

In the next results we shall use the function
\begin{equation}\label{eq:G}
G(z)=\sum_{n=0}^\infty \frac{z^n}{b_n^n},
\end{equation}
which is entire if $b_n\to\infty.$

\begin{lem}\label{thm:3term}
Suppose that the recurrence coefficients of \eqref{eq:3t} satisfy
\begin{enumerate}
\item[(i)] $a_n=O(b_n)$,
\item[(ii)] $(b_{n})$ is eventually increasing,
\item[(iii)] $b_n\to\infty$.
\end{enumerate}
Then there exist constants $A,C\ge 1$ such that
\begin{equation}\label{eq:s2nup}
\sqrt{s_{2n}}\le A(3C)^nb_0b_1\cdots b_{n-1}, n\ge 0.
\end{equation}
\end{lem}

\begin{proof} Because of the assumption (i) there exists a constant $C \ge 1$ such that $|a_n|\le Cb_n$ for all $n\ge 0$.
By (ii) there exists $n_0\ge 1$ such that $b_{n-1}\le b_n$ for $n\ge n_0$ and by (iii) there exists $n_1\ge n_0$ such that
$b_{n_1}\ge\max(1,b_0,\ldots,b_{n_0-1})$, hence
\begin{equation}\label{eq:Bmax}
B:=\max(1,b_0,\ldots,b_{n_1-1})\le b_{n_1}.
\end{equation}
The three term recurrence relation \eqref{eq:3t} for $P_n$ applied successively leads to
\begin{eqnarray*}
x &=& a_0P_0+b_0P_1,\\
x^2 &=& x(a_0P_0+b_0P_1)=a_0(a_0P_0+b_0P_1)+b_0(b_0P_0+a_1P_1+b_1P_2),
\end{eqnarray*}  
and in general there exist  an index set $I_n$ with $|I_n|\le 3^n$, a mapping $J_n$ from $I_n$ to $\{0,1,\ldots, n\}$ and real coefficients
$d_{n,k},k\in I_n$ such that
\begin{equation}\label{eq:l}
x^n=\sum_{k\in I_n} d_{n,k}P_{J_n(k)}.
\end{equation} 

In the next step we get
$$
x^{n+1}=\sum_{k\in I_n} d_{n,k}\left(b_{J_n(k)-1}P_{J_n(k)-1} + a_{J_n(k)}P_{J_n(k)} + b_{J_n(k)} P_{J_n(k)+1}\right),
$$
which shows how each element $k\in I_n$ gives rise to two or three elements in $I_{n+1}$ depending on $J_n(k)=0$ or $J_n(k)>0$.

Each $d_{n,k}$ is a product of $n$ terms from $\{a_0,\ldots,a_{n-1},b_0,\ldots,b_{n-1}\}$, hence
$$
|d_{n,k}|\le C^n(\max(b_0,\ldots,b_{n-1}))^n.
$$
For $n\le n_1$ we have in particular $|d_{n,k}|\le (BC)^n\le B^{n_1}C^n$.

 We claim that in general 
\begin{equation}\label{eq:dnk}
|d_{n,k}|\le B^{n_1}C^n b_{n_1}\cdots b_{n-1},\quad k\in I_n,\;n\ge 1,
\end{equation}
which is already established for $n\le n_1$, where the empty product $b_{n_1}\cdots b_{n-1}$ is to be understood as 1.
Assume now that \eqref{eq:dnk} holds for some $n\ge n_1$. If $J_n(k)\ge n_1$ we have
\begin{eqnarray*}
|d_{n,k}|b_{J_n(k)-1} &\le& |d_{n,k}|b_{J_n(k)}\le |d_{n,k}|b_n\le B^{n_1}C^nb_{n_1}\cdots b_{n-1}b_n\\ 
 |d_{n,k}||a_{J_n(k)}| &\le& C|d_{n,k}|b_{J_n(k)}\le B^{n_1} C^{n+1}b_{n_1}\cdots b_{n-1}b_n,
\end{eqnarray*}
while if $J_n(k)\le n_1-1$
\begin{eqnarray*}
|d_{n,k}|b_{J_n(k)-1}, |d_{n,k}|b_{J_n(k)} &\le& |d_{n,k}|B \le  B^{n_1}C^nb_{n_1}\cdots b_{n-1}b_n\\ 
  |d_{n,k}||a_{J_n(k)}| &\le& C|d_{n,k}|b_{J_n(k)}\le  B^{n_1}C^{n+1}b_{n_1}\cdots b_{n-1}b_n,
\end{eqnarray*}
where we have used that $B\le b_{n_1}\le b_n$. This finishes the induction proof of \eqref{eq:dnk},
which may be written 
$$
|d_{n,k}|\le AC^nb_0b_1\cdots b_{n-1},\quad k\in I_n\; n\ge 1,
$$
where $A=B^{n_1}/(b_0b_1\cdots b_{n_1-1})$.

Now \eqref{eq:s2nup} follows because 
\begin{eqnarray*}
s_{2n} &=&\int x^{2n}\,d\mu(x) =\sum_{k\in I_n}\sum_{l\in I_n}
d_{n,k}d_{n,l} \int P_{J_n(k)}P_{J_n(l)}\,d\mu(x)\\
&\le& \sum_{k\in I_n}\sum_{l\in I_n}
|d_{n,k}|\,|d_{n,l} |=\left (\sum_{k\in I_n}|d_{n,k}|\right)^2\le \left(3^nAC^{n}b_0b_1\cdots b_{n-1}\right)^2.
\end{eqnarray*}
\end{proof}

\begin{prop}\label{thm:3term1} Let $(s_n)$ denote an indeterminate moment sequence for which the recurrence
 coefficients \eqref{eq:3t} satisfy the conditions of Lemma~\ref{thm:3term}.
Then 
\begin{enumerate}
\item[(i)] $\rho_G \le \rho_L=\rho_H$.
\item[(ii)] $\w{G} \le \w{L}=\w{H}$, provided $\rho_H=0$.
\item[(iii)] $\ww{G} \le \ww{L}=\ww{H}$, provided $\w{H}=0$.
\end{enumerate}
\end{prop}

\begin{proof}
From \eqref{eq:s2nup}, \eqref{eq:lead} and $b_{n-1}\le b_n$ for $n\ge n_1$,
it follows for such $n$ that
 $$
\sqrt{s_{2n}} \le \frac{A(3C)^n}{b_{n,n}} \le B^{n_1}(3C)^n b_n^{n-n_1},
$$
where $B$ is given by \eqref{eq:Bmax},
hence
 $$
\sqrt{s_{2n}} \le \frac{\alpha (3C)^n}{b_{n,n}} \le \gamma (3C)^n b_n^{n-n_1},\quad n\ge 0,
$$
for suitable constants $\alpha,\gamma>0$. Introducing
$$
G^*(z)=\sum_{n=0}^\infty \frac{1}{b_n^{n-n_1}}z^n,
$$
 this gives
 $$
M_L(r)\ge (1/\alpha) M_H(r/3C)\ge (1/\gamma)M_{G^*}(r/3C),\quad r>0,
$$
showing that
$\rho_L \ge \rho_H \ge \rho_G^*$ and similar inequalities for the logarithmic and double logarithmic orders.
If this is combined with Proposition~\ref{thm:newLif}, we get the equality sign between the orders of $L$ and $H$.
Furthermore, by \eqref{eq:deforder}
$$
\rho_{G^*}=\limsup\frac{\log n}{(1-n_1/n)\log b_n}=\limsup\frac{\log n}{\log b_n}=\rho_G,
$$
and similarly $\w{G}=\w{G^*}$ and  $\ww{G}=\ww{G^*}$.
\end{proof}

\begin{thm}\label{thm:=inLif1} Given an (indeterminate) moment problem where
$$
\sum_{n=1}^\infty \frac{1+|a_n|}{\sqrt{b_nb_{n-1}}}<\infty,
$$
and where either \eqref{eq:logconvex} or \eqref{eq:logconcave} holds.

The following holds 
\begin{enumerate}
\item[(i)] $\rho=\rho_F=\rho_G=\rho_H=\rho_L=\mathcal E(b_n)$.

If $\rho=0$ then
\item[(ii)] $\w{}=\w{F}=\w{G}=\w{H}=\w{L}=\mathcal E(\log b_n)$.

If $\w{}=0$ then
\item[(iii)] $\ww{}=\ww{F}=\ww{G}=\ww{H}=\ww{L}=\mathcal E(\log\log b_n)$.
\end{enumerate}
\end{thm}   

\begin{proof}
By Lemma~\ref{thm:inc} we know that $b_{n-1}\le b_n$ for $n\ge n_1$, so the conditions of Proposition~\ref{thm:3term1} are fulfilled.
By \eqref{eq:consBer} we have
$$
\frac{1}{b_n^{2n}} \le \frac{1}{b_{n-1}^{2n}}\le Ab_{n,n}^2,\quad n\ge n_1
$$
for a certain constant $A$, and by replacing $A$ by a larger constant
if necessary, we see that there exists a constant $a$ such that $1/b_n^n\le a b_{n,n}$ for all
$n$. This gives $M_G(r)\le a M_H(r)$, hence $\rho_G\le\rho_H$. By
\eqref{eq:deforder} we have
$$
\rho_G=\limsup_{n\to\infty}\frac{\log n}{\log b_n},
$$
so for any $\varepsilon>0$ we get $n\le b_n^{\rho_G+\varepsilon}$ for
$n$ sufficiently large. This gives
$$
\sum_{n=0}^\infty
\frac{1}{b_n^{(\rho_G+\varepsilon)(1+\varepsilon)}}<\infty,
$$
hence $\mathcal E(b_n)\le \rho_G$. Finally, by
Theorem~\ref{thm:beralpha1}, Proposition~\ref{thm:newLif} and Proposition~\ref{thm:3term1} we get  $\rho=\mathcal E(b_n)\le\rho_G\le\rho_H=\rho_L\le\rho$.

If the common order $\rho=0$, we get as above
$\w{G}\le\w{H}$, and by \eqref{eq:logorder}
we know that
$$
\w{G}=\limsup_{n\to\infty} \frac{\log n}{\log\log b_n}.
$$
For given $\varepsilon>0$ we get for $n$ sufficiently large that
$$
n\le (\log b_n)^{\w{G}+\varepsilon},
$$
showing that $\mathcal E(\log b_n)\le \w{G}$. We finally use
Theorem~\ref{thm:Berlogorder} combined with Proposition~\ref{thm:newLif} and Proposition~\ref{thm:3term1} to get (ii), and proceed similarly concerning the double logarithmic order.
\end{proof}

\begin{ex}\label{thm:chenismail} {\rm In \cite{C:I} symmetric polynomials with the recurrence coefficients
$b_{n-1}=2n\sqrt{4n^2-1},n\ge 1$, are considered. The sequence is log-concave and the order of the moment problem is $1/2$ by Theorem~\ref{thm:beralpha1}.

The case of $b_{n-1}=q^{-n}$ for $0<q<1$ is also considered, and Chen and Ismail  find explicit representations of $P_n$ and the entire functions $A,B,C,D$.
Clearly $b_n^2=b_{n-1}b_{n+1}$ and we find that the order is 0 and the logarithmic order is 1 in accordance with the estimates of the paper.}  
\end{ex}

\section{Appendix}\label{sec-app}

{\it Proof of Theorem~\ref{thm:dlogcoef}}. To establish  \eqref{eq:dlogorder}, we first show that if
$$
M_f(r)\le r^{(\log\log r)^\alpha} ,\quad \alpha>0,r\ge r_0,
$$
then
\begin{equation}\label{eq:first}
 \limsup_{n\to \infty}\frac{\log n}{\log\log \log\left(
    \frac{1}{\sqrt[n]{|a_n|}}\right)}\le \alpha.
\end{equation}
This will yield $\ge$ in \eqref{eq:dlogorder}.

By the Cauchy estimates
$$
|a_n|\le \frac{M_f(r)}{r^n}\le r^{(\log\log r)^\alpha-n},\quad r\ge
r_0.
$$
In this inequality we will choose an $r$ approximately minimizing 
$$
\varphi(r)=((\log\log r)^\alpha-n)\log r.
$$
Note that $\varphi'(r)=0$ if $x=\log\log r$ satisfies
\begin{equation}\label{eq:alphahelp}
x^\alpha+\alpha x^{\alpha-1}-n=0.
\end{equation}
Motivated by Lemma~\ref{thm:alphahelp} below 
we choose $r$ such that $\log\log r=n^{1/\alpha}-1$. This is certainly larger than $r_0$ if $n$ is large
enough. Inserting this value for $r$, we get
$$
\log|a_n|\le ((n^{1/\alpha}-1)^\alpha
-n)\exp(n^{1/\alpha}-1)=-n\left(1-(1-n^{-1/\alpha})^\alpha\right)
\exp(n^{1/\alpha}-1),
$$
hence
$$
\log\log \frac{1}{\sqrt[n]{|a_n|}} \ge n^{1/\alpha}-1+\log\left(1-(1-n^{-1/\alpha})^\alpha\right)
=n^{1/\alpha}(1+o(1)),
$$
showing \eqref{eq:first}.

We next show that the double logarithmic order of $f$ satisfies

\begin{eqnarray}\label{eq:dlogorder1}
\ww{f} 
\le \limsup_{n\to \infty}\frac{\log n}{\log\log \log\left(
    \frac{1}{\sqrt[n]{|a_n|}}\right)}.
\end{eqnarray}

This is clear if the right-hand side is infinity. Let $\mu$ be an arbitrary number larger than the right-hand side, now assumed finite.
Then there exists $n_0$ such that
$$
\log n\le \mu\log\log\log  \frac{1}{\sqrt[n]{|a_n|}},\quad n\ge n_0,
$$
or 
$$
|a_n|\le \exp\left(-n\exp(n^{1/\mu})\right),\quad n\ge n_0.
$$
Fix $r>e$ so large that $\log r> \exp(n_0^{1/\mu})-1$. We next determine $n_1>n_0$ so that
$$
 \exp((n_1-1)^{1/\mu})-1 < \log r \le \exp(n_1^{1/\mu})-1.
$$
For this $r$ we find with $C_1=\sum_{n=0}^{n_0-1}|a_n|$
 \begin{eqnarray*}
M_f(r) &\le & \sum_{n=0}^{n_0-1} |a_n|r^n + \sum_{n=n_0}^{\infty} |a_n|r^n\\
& \le & C_1 r^{n_0} + \sum_{n=n_0}^{\infty}\exp\left(-n\exp(n^{1/\mu})+n\log r\right)\\
&\le & C_1 r^{n_0} + \sum_{n=n_0}^{n_1-1}\exp\left(-n\exp(n^{1/\mu})+(\log(1+\log r))^\mu\log r\right)\\
& + &  
\sum_{n=n_1}^{\infty}\exp\left(-n\exp(n^{1/\mu})+n\exp(n^{1/\mu})-n\right),
\end{eqnarray*}
where we have used in the second sum that for $n_0\le n<n_1$:
$\exp(n^{1/\mu})-1< \log r$, hence $n< (\log(1+\log r))^\mu$, and in the last sum that for $n\ge n_1$
$$
\log r\le \exp(n_1^{1/\mu})-1\le \exp(n^{1/\mu})-1.
$$
We then get
 \begin{eqnarray*}
M_f(r) &\le &  C_1 r^{n_0} + r^{(\log(1+\log r))^{\mu}}\sum_{n=n_0}^{n_1-1}\exp\left(-n\exp(n^{1/\mu})\right)+\sum_{n=n_1}^\infty\exp(-n)\\
 &< & C_1 r^{n_0} + r^{(\log(1+\log r))^{\mu}}+1,
\end{eqnarray*}
where we have majorized the two sums by $\sum_1^\infty\exp(-n)=1/(e-1)<1$.
For any given $\varepsilon>0$ we have
$$
(\log(1+\log r))^{\mu} \le_{\text{\tiny as}} (\log\log r)^{\mu+\varepsilon},
$$
hence
$$
M_f(r)\le_{\text{\tiny as}} 2 r^{(\log\log r)^{\mu+\varepsilon}}
\le_{\text{\tiny as}} r^{(\log\log r)^{\mu+2\varepsilon}}.
$$
This  establishes $\ww{f}\le \mu + 2\varepsilon$, which shows  $\le$ in \eqref{eq:dlogorder}. 

We next prove \eqref{eq:dlogtype}. For simplicity of notation we put $\alpha=\ww{f}$ and assume that $0<\alpha<\infty$. We show first that if
$$
M_f(r)\le r^{K(\log\log r)^\alpha} ,\quad K>0,\;r\ge r_0,
$$
then
\begin{equation}\label{eq:second}
 \limsup_{n\to \infty}\frac{n}{\left(\log \log
    \frac{1}{\sqrt[n]{|a_n|}}\right)^\alpha}\le K,
\end{equation}
which establishes $\ge$ in \eqref{eq:dlogtype}.

By the Cauchy estimates
$$
|a_n|\le \frac{M_f(r)}{r^n}\le r^{K(\log\log r)^\alpha-n},\quad r\ge
r_0,
$$
hence
$$
\log|a_n|\le \left(K(\log\log r)^\alpha-n\right)\log r,\quad r\ge r_0.
$$

In this inequality we will choose  $\log\log r=(n/K)^{1/\alpha}-1$ by inspiration from the proof in the first part. This gives
$$
\log|a_n|\le -n\left(1-[1-(n/K)^{-1/\alpha}]^\alpha\right)
\exp((n/K)^{1/\alpha}-1),
$$
hence
$$
\log\log \frac{1}{\sqrt[n]{|a_n|}} \ge (n/K)^{1/\alpha}-1+\log\left(1-[1-(n/K)^{-1/\alpha}]^\alpha\right)
=(n/K)^{1/\alpha}(1+o(1)),
$$
showing \eqref{eq:second}.

We next show that the double logarithmic type of $f$ satisfies

\begin{eqnarray}\label{eq:dlogtype1}
\dd{f} 
\le \limsup_{n\to \infty}\frac{n}{\left(\log \log
    \frac{1}{\sqrt[n]{|a_n|}}\right)^\alpha}.
\end{eqnarray}

This is clear if the right-hand side is infinity. Let $\mu$ be an arbitrary number larger than the right-hand side, now assumed finite.
Then there exists $n_0$ such that
$$
 n\le \mu\left(\log\log  \frac{1}{\sqrt[n]{|a_n|}}\right)^\alpha,\quad n\ge n_0,
$$
or
$$
|a_n|\le \exp\left(-n\exp((n/\mu)^{1/\alpha})\right),\quad n\ge n_0.
$$
Fix $r>e$ so large that $\log r> \exp((n_0/\mu)^{1/\alpha})-1$. We next determine $n_1>n_0$ so that
$$
 \exp\left(\left(\frac{n_1-1}{\mu}\right)^{1/\alpha}\right)-1 < \log r \le \exp((n_1/\mu)^{1/\alpha})-1.
$$
For this $r$ we find with $C_1=\sum_{n=0}^{n_0-1}|a_n|$
 \begin{eqnarray*}
M_f(r) & \le & C_1 r^{n_0} + \sum_{n=n_0}^{\infty}\exp\left(-n\exp((n/\mu)^{1/\alpha})+n\log r\right)\\
&\le & C_1 r^{n_0} + \sum_{n=n_0}^{n_1-1}\exp\left(-n\exp((n/\mu)^{1/\alpha})+\mu(\log(1+\log r))^\alpha\log r\right)\\
& + &  
\sum_{n=n_1}^{\infty}\exp(-n),
\end{eqnarray*}
where we have used that $n<\mu\left(\log(1+\log r)\right)^{\alpha}$ when $n_0\le n\le n_1-1$, and that $\log r\le \exp\left((n/\mu)^{1/\alpha}\right)-1$ when
$n\ge n_1$.

We then get
 \begin{eqnarray*}
M_f(r) &\le &  C_1 r^{n_0} + r^{\mu(\log(1+\log r))^{\alpha}}\sum_{n=n_0}^{n_1-1}\exp\left(-n\exp((n/\mu)^{1/\alpha}\right)+\sum_{n=n_1}^\infty\exp(-n)\\
 &< & C_1 r^{n_0} + r^{\mu(\log(1+\log r))^{\alpha}}+1.
\end{eqnarray*}

For any given $\varepsilon>0$ we have
$$
\mu(\log(1+\log r))^{\alpha}\le_{\text{\tiny as}} (\mu+\varepsilon)(\log\log r)^{\alpha},
$$
hence
$$
M_f(r))\le_{\text{\tiny as}} 2 r^{(\mu+\varepsilon)(\log\log r)^{\alpha}}
\le_{\text{\tiny as}} r^{(\mu+2\varepsilon)(\log\log r)^{\alpha}}.
$$
This  establishes $\dd{f}\le \mu + 2\varepsilon$, which shows $\le$ in \eqref{eq:dlogtype}. 

\hfill $\square$

\begin{lem}\label{thm:alphahelp} Let $n\in \mathbb N, n\ge 4$ and $\alpha>0$. Then
  the function in \eqref{eq:alphahelp}
$$
h(x)=x^\alpha+\alpha x^{\alpha-1}-n
$$
has a zero in 
$$
\begin{cases}
  \displaystyle [n^{1/\alpha}-1,n^{1/\alpha}]  & \text{if \;$\alpha >1$},\\
\displaystyle n-1 & \text{if \;$\alpha=1$},\\
\displaystyle [n^{1/\alpha}-2,n^{1/\alpha}-1] & \text{if \;$0<\alpha<1$}.
\end{cases}
$$
\end{lem}

\begin{proof} We find $h(n^{1/\alpha})=\alpha n^{1-1/\alpha}>0$ for
  all $\alpha>0$. Putting $y=n^{1/\alpha}-1$ we find for some $\xi\in (0,1)$
$$
(y+1)^\alpha-y^\alpha= \alpha(y+\xi)^{\alpha-1}\;\begin{cases}
  \displaystyle > \alpha y^{\alpha-1} & \text{if \; $\alpha>1$},\\
\displaystyle < \alpha y^{\alpha-1} & \text{if \; $0<\alpha <1$}.
\end{cases}
$$
This shows that $h(n^{1/\alpha}-1) < 0$ (resp. $>0$) for $\alpha>1$
(resp. $0<\alpha<1$).

Finally, for $0<\alpha<1$ we put $y=n^{1/\alpha}-2$ and get for some $0<\eta<2$
$$
(y+2)^{\alpha}-y^\alpha=2\alpha (y+\eta)^{\alpha-1}>\alpha
y^{\alpha-1}
$$
if $y\ge 2$. This shows that
 $h(n^{1/\alpha}-2)<0$. Note that $y=n^{1/\alpha}-2\ge 2$ for $n\ge 4$.
\end{proof}

Propositions 5.3 and 5.4 from \cite{B:P:H} can be extended to double logarithmic order.

These results deal with transcendental entire functions $f$ of ordinary order strictly less than $1$. They have
infinitely many zeros, which we label $\{ z_n\}$ and number according
to increasing order of magnitude. We repeat each zero according to
its multiplicity. Supposing $f(0)=1$ we get from Hadamard's
factorization theorem 
\begin{equation}\label{eq:canonical}
f(z)=\prod_{n=1}^{\infty}\left( 1 - \frac{z}{z_n}\right) .
\end{equation}

The growth of $f$ is
thus determined by the distribution of the zeros. We shall use the following
quantities to describe this distribution.

The usual zero
counting function $n(r)$ is 
$$
n(r)= \# \{ n \, | \, |z_n|\leq r\},
$$
and we define  
$$
N(r)=\int_0^r\frac{n(t)}{t}\, dt,
$$
and 
$$
Q(r)=r\int_r^{\infty}\frac{n(t)}{t^2}\, dt.
$$
These quantities are related to $M_f(r)$ in the following way
\begin{equation}
\label{eq:MNQ}
N(r)\leq \log M_f(r)\leq N(r)+Q(r)
\end{equation}
for $r>0$. (This is relation (3.5.4) in Boas \cite{BOAS}).

By a theorem of Borel it is known that $\rho_f=\mathcal E(z_n)$, and if the order is $0$, then $\rho_f^{[1]}=\mathcal E(\log|z_n|)$ by Proposition 5.4
in \cite{B:P:H}. Furthermore, by Proposition 5.3 in \cite{B:P:H} we have
$$
\mathcal E(\log|z_n|)=\limsup_{n\to\infty} \frac{\log n(r)}{\log\log r}.
$$

The following proposition expresses the double logarithmic convergence
exponent $\mathcal E(\log\log|z_n|)$ in terms of the zero counting
function of $f$.

\begin{prop}
\label{thm:dlc}
We have 
\begin{equation}\label{eq:Lhelp}
\mathcal E(\log\log |z_n|) = \limsup_{r\to \infty}\frac{\log n(r)}{\log \log \log r}.
\end{equation}
\end{prop}

\begin{proof} We have
$$
n(e^{e^r})=\#\left\{n\, \mid |z_n|\le e^{e^r}\right\} =\#\{n\, \mid \log\log|z_n|\le r\},
$$
hence by Lemma~\ref{thm:expc=or} 
$$
\mathcal E(\log\log|z_n|)=\limsup_{r\to\infty}\frac{\log n(e^{e^r})}{\log r}=\limsup_{s\to\infty}\frac{\log n(s)}{\log\log\log s}.
$$
\end{proof}

\begin{thm}\label{thm:dlc1} The double logarithmic order of the canonical product \eqref{eq:canonical} is equal to the double logarithmic 
convergence exponent of the zeros, i.e., $\ww{f}=\mathcal E(\log\log |z_n|)$.
\end{thm}
  
\begin{proof} We shall prove that $L=\ww{f}$, where $L$ is given by the right-hand side of \eqref{eq:Lhelp}. 
Let $\alpha>0$ be such that
$$
M_f(r)\le r^{(\log\log r)^{\alpha}},\quad r\ge r_0.
$$
For $r\ge r_0$ we then get by the left-hand side of \eqref{eq:MNQ}
$$ 
n(r)\log r \le \int_r^{r^2}\frac{n(t)}t\,dt \le N(r^2) \le \log M_f(r^2) \le 2(\log\log r^2)^\alpha\log r,
$$
hence for any $\varepsilon >0$
$$
n(r) \le 2(\log 2 +\log\log r)^\alpha  \le_{\text{\tiny as}}(\log\log r)^{\alpha+\varepsilon},
$$
which shows that $L\le \alpha+\varepsilon$, leading to $L\le \ww{f}$.

To prove the converse inequality we let $\varepsilon>0$ be given. There exists $r_0>1$ such that
$$
n(r) \le (\log\log r)^{L+\varepsilon},\quad r\ge r_0.
$$
For $r>r_0$ we then get
$$
N(r) \le \int_0^{r_0}\frac{n(t)}t\,dt + \int_{r_0}^r(\log\log t)^{L+\varepsilon}\frac{dt}{t}< \int_0^{r_0}\frac{n(t)}t\,dt + (\log\log r)^{L+\varepsilon}\log r.
$$ 
We also get
$$
Q(r)\le r\int_r^{\infty}\frac{(\log\log t)^{L+\varepsilon}}{t^{1/2}}\,\frac{dt}{t^{3/2}}.
$$
We next use that 
$$
\frac{t^{1/2}}{(\log\log t)^{L+\varepsilon}}=\left[\frac{t}{(\log\log t)^{2(L+\varepsilon)}}\right]^{1/2}
$$
is increasing for $t$ sufficiently large, because $(\log\log r)^{\alpha}$ is an order function for any $\alpha>0$. We can therefore write
$$
Q(r) \le r \frac{(\log\log r)^{L+\varepsilon}}{r^{1/2}}\int_r^\infty\,\frac{dt}{t^{3/2}}=2(\log\log r)^{L+\varepsilon},
$$
so  by the right-hand side of \eqref{eq:MNQ} we find
$$
\limsup_{r\to\infty} \frac{\log M_f(r)}{(\log\log r)^{L+\varepsilon}\log r} \le 1,
$$
and it follows that $\ww{f}\le L$.
\end{proof}

\medskip
\noindent{\bf Acknowledgment.} The authors thank Henrik Laurberg Pedersen for valuable comments to the manuscript.

Christian Berg; email:berg@math.ku.dk\\
Department of Mathematics, University of Copenhagen,\\
Universitetsparken 5, DK-2100, Denmark

\medskip
Ryszard Szwarc; email szwarc2@gmail.com\\
Institute of Mathematics,
University of Wroc{\l}aw, pl.\ Grunwaldzki 2/4, 50-384 Wroc{\l}aw, Poland 
 
and

\noindent Institute of Mathematics and Computer Science, 
University of Opole, ul. Oleska 48,
45-052 Opole, Poland

\end{document}